\documentclass[12pt]{amsart}
\usepackage{amssymb}
\usepackage{amsfonts}
\usepackage{latexsym}
\usepackage{amscd}
\usepackage[mathscr]{euscript}
\usepackage{xy} \xyoption{all}

\newcommand{\N}{{\mathbb{N}}}

\newcommand{\Z}{{\mathbb{Z}}}

\newcommand{\uloopr}[1]{\ar@'{@+{[0,0]+(-4,5)}@+{[0,0]+(0,10)}@+{[0,0] +(4,5)}}^{#1}}
\newcommand{\uloopd}[1]{\ar@'{@+{[0,0]+(5,4)}@+{[0,0]+(10,0)}@+{[0,0]+ (5,-4)}}^{#1}}
\newcommand{\dloopr}[1]{\ar@'{@+{[0,0]+(-4,-5)}@+{[0,0]+(0,-10)}@+{[0, 0]+(4,-5)}}_{#1}}
\newcommand{\dloopd}[1]{\ar@'{@+{[0,0]+(-5,4)}@+{[0,0]+(-10,0)}@+{[0,0 ]+(-5,-4)}}_{#1}}

\newcommand{\luloop}[1]{\ar@'{@+{[0,0]+(-8,2)}@+{[0,0]+(-10,10)}@+{[0, 0]+(2,2)}}^{#1}}

\newtheorem{lem}{Lemma}[section]
\newtheorem{corol}[lem]{Corollary}
\newtheorem{theor}[lem]{Theorem}
\newtheorem{prop}[lem]{Proposition}

\newtheorem{defi}[lem]{Definition}
\newtheorem{defis}[lem]{Definitions}
\newtheorem{exem}[lem]{Example}

\begin{document}
\title[Flow invariants for Leavitt path algebras]{Flow invariants in the
classification of \\ Leavitt path algebras}
\author{Gene Abrams}
\address{Department of Mathematics, University of Colorado,
Colorado Springs CO 80918 U.S.A.} \email{abrams@math.uccs.edu,  cdsmith@gmail.com}
\author{Adel Louly}\author{Enrique Pardo}
\address{Departamento de Matem\'aticas, Facultad de Ciencias, Universidad de C\'adiz,
Campus de Puerto Real, 11510 Puerto Real (C\'adiz),
Spain.}\email{louly.adel@uca.es,
enrique.pardo@uca.es}
 \author{Christopher Smith}

\thanks{The first author is partially supported by the U.S. National Security Agency under grant number H89230-09-1-0066. The second author was partially supported by a Post-doctoral fellow of the A.E.C.I. (Spain). The
third author was partially supported by the DGI and European Regional Development Fund, jointly,
through Projects MTM2007-60338 and MTM2008-06201-C02-02, by the Consolider Ingenio
``Mathematica" project CSD2006-32 by the MEC, by PAI III grants FQM-298 and P06-FQM-1889 of the Junta de Andaluc\'{\i}a, and by the Comissionat per Universitats i Recerca de la Generalitat de
Catalunya.} 
%

\maketitle

\begin{abstract}
We analyze in the context of Leavitt path algebras some graph
operations introduced in the context of symbolic dynamics by
Williams, Parry and Sullivan, and Franks. We show that these
operations induce Morita equivalence of the corresponding Leavitt
path algebras. As a consequence we obtain our two main results:  the
first gives sufficient conditions for which the Leavitt path
algebras in a certain class are Morita equivalent, while the second
gives sufficient conditions which yield isomorphisms.    We discuss
a possible approach to establishing whether or not these conditions
are also in fact necessary.  In the final section we present many
additional operations on graphs which preserve Morita equivalence
(resp., isomorphism) of the corresponding Leavitt path algebras.
\end{abstract}

\section*{Introduction}\label{Introduction}

Throughout this article $E$ will denote a row-finite directed graph,
and $K$ will denote an arbitrary field.  The {\it Leavitt path
algebra of $E$ with coefficients in $K$}, denoted $L_K(E)$, has
received significant attention  over the past few years, both from
algebraists as well as from analysts working in operator theory.
(The precise definition of $L_K(E)$ is given below.)   When $K$ is
the field ${\mathbb C}$ of complex numbers, the algebra $L_K(E)$ has
exhibited surprising similarity to C$^*(E)$, the {\it graph}
C$^*$-{\it algebra of} $E$.  In this context, it is natural to ask
whether an analog of the Kirchberg-Phillips Classification Theorem
\cite{Kirch, Phil} for C$^*$-algebras holds for various classes of
Leavitt path algebras as well.  Specifically, the following question
was posed in \cite{AbAnhLouP}:

\bigskip

\textbf{The Classification Question for purely infinite simple
unital Leavitt path algebras:} Let $K$ be a field, and suppose $E$
and $F$ are graphs for which $L_K(E)$ and $L_K(F)$ are purely
infinite simple unital. If $K_0(L_K(E)) \cong K_0(L_K(F))$ via an
isomorphism $\varphi$ having $\varphi([1_{L_K(E)}]) = [1_{L_K(F)}]$, must
$L_K(E)$ and $L_K(F)$ be isomorphic?

\bigskip

The Classification Question is answered in the affirmative in
\cite{AbAnhLouP} for a few specific classes of graphs.  We obtain in
the current article an affirmative answer for a significantly wider
class of graphs.   Our approach is as follows. In Section
\ref{Morita Equiv} we consider Morita equivalence of Leavitt path
algebras.  By applying a deep theorem of Franks \cite{Franks} from
the field of symbolic dynamics, we obtain in Theorem
\ref{Franksinvariantsufficient} a sufficient set of conditions on
$E$ and $F$ which ensure that $L_K(E)$ is Morita equivalent to
$L_K(F)$.  (Ideas from symbolic dynamics were employed in analyzing structures related to Leavitt path algebras in, for instance, \cite{CK}; we describe these more fully below.)  In Section \ref{Isomorphism}, we exploit these Morita
equivalences to obtain sufficient conditions which ensure
isomorphism (Theorem \ref{Th:Almost BigFish}), thereby obtaining the
aforementioned partial affirmative answer to the Classification
Question.

We complete Section \ref{Isomorphism} by examining  the remaining
difficulty in obtaining an affirmative answer to the Classification
Question for {\it all} germane graphs.  In Section \ref{Extra} we
extend several results about Morita equivalence and isomorphism to
certain classes of graphs $E$ for which $L_K(E)$ is not necessarily
purely infinite simple unital, thereby giving more general results
than have been previously known about isomorphism and Morita
equivalence of Leavitt path algebras.

We briefly recall some graph-theoretic definitions and properties;
more complete explanations and descriptions can be found in
\cite{AA1}. A \emph{graph} (synonymously, a \emph{directed graph})
$E=(E^0,E^1,r_E,s_E)$ consists of two sets $E^0,E^1$ and maps $r_E,s_E:E^1
\to E^0$.  The elements of $E^0$ are called \emph{vertices} and the
elements of $E^1$ \emph{edges}. We write $s$ for $s_E$ (resp., $r$ for $r_E$) if the graph $E$ is clear from context.  We emphasize that loops and multiple /
parallel edges are allowed.  If $s^{-1}(v)$ is a finite set for
every $v\in E^0$, then the graph is called \emph{row-finite}.  All
graphs in this paper will be assumed to be row-finite.  A vertex
$v$ for which $s^{-1}(v)$ is empty is called a \emph{sink};  a
vertex $w$ for which $r^{-1}(w)$ is empty is called a \emph{source}.

A \emph{path} $\mu$ in a graph $E$ is either a vertex, or a sequence of edges
$\mu=e_1\dots e_n$ such that $r(e_i)=s(e_{i+1})$ for
$i=1,\dots,n-1$. In the latter case, $s(\mu):=s(e_1)$ is  the \emph{source}
of $\mu$, $r(\mu):=r(e_n)$ is the \emph{range} of $\mu$, and $n$ is
the \emph{length} of $\mu$.  If $\mu = v$ is a vertex, we define $s(v) = r(v)=v$, and define the length of $v$ to be $0$.   An edge $f$ is an {\it exit} for a path
$\mu = e_1 \dots e_n$ if there exists $i$ such that $s(f)=s(e_i)$
and $f \neq e_i$. If $\mu$ is a path in $E$, and if
$v=s(\mu)=r(\mu)$, then $\mu$ is called a \emph{closed path based at
$v$}. If $\mu= e_1 \dots e_n$ is a closed path based at $v = s(\mu)$
and $s(e_i)\neq s(e_j)$ for every $i\neq j$, then $\mu$ is called a
\emph{cycle}.

The following notation is standard.  Let $A$ be a $p\times p$ matrix
having non-negative integer entries (i.e., $A = (a_{ij})\in {\rm
M}_p(\Z^+)$).  The graph $E_A$ is defined by setting
$(E_A)^0=\{v_1,v_2,\ldots,v_p\}$, and defining $(E_A)^1$ by
inserting exactly $a_{ij}$ edges in $E_A$ having source vertex $v_i$
and range vertex $v_j$.  Conversely, if $E$ is a finite graph with
vertices $\{v_1,v_2,...,v_p\}$, then we define the {\it incidence
matrix} $A_E$ {\it of} $E$ by setting $(A_E)_{ij}$ as the number of
edges in $E$ having source vertex $v_i$ and range vertex $v_j$.

Given a graph $E = (E^0,E^1,r,s)$, we define the {\it transpose
graph} $E^t$ to be the graph $(E^0,E^1,s,r)$ with the same vertices
as $E$, but with edges in the opposite direction.  Notice that
$A_{E^t}=(A_E)^t$, and $E_{A^t} = (E_A)^t$, as implied by the
notation.

Our focus in this article is on $L_K(E)$, the Leavitt path algebra
of $E$. We define $L_K(E)$  here, after which we review some
important properties and examples.

\begin{defi}\label{definition}  {\rm Let $E$ be any row-finite graph, and $K$ any field.
The {\em Leavitt path $K$-algebra} $L_K(E)$ {\em of $E$ with coefficients in $K$} is
the $K$-algebra generated by a set $\{v \mid v\in E^0\}$ of pairwise orthogonal idempotents,
together with a set of variables $\{e,e^* \mid e \in E^1 \}$, which satisfy the following
relations:

(1) $s(e)e=er(e)=e$ for all $e\in E^1$.

(2) $r(e)e^*=e^*s(e)=e^*$ for all $e\in E^1$.

(3) (The ``CK1 relations") \ $e^*e'=\delta _{e,e'}r(e)$ for all $e,e'\in E^1$.

(4) (The ``CK2 relations") \ $v=\sum _{\{ e\in E^1\mid s(e)=v \}}ee^*$ for every vertex
$v\in E^0$ for which $s^{-1}(v)$ is
nonempty.
}
\end{defi}

When the role of the coefficient field $K$ is not central to the
discussion, we will often denote $L_K(E)$ simply by $L(E)$. The set
$\{e^*\mid e\in E^1\}$ will be denoted by $(E^1)^*$. We let $r(e^*)$
denote $s(e)$, and we let $s(e^*)$ denote $r(e)$. If $\mu = e_1
\dots e_n$ is a path, then we denote by $\mu^*$ the element $e_n^*
\cdots e_1^*$ of $L_K(E)$.

An alternate description of $L_K(E)$ is given in \cite{AA1}, where
it is described in terms of a free associative algebra modulo the
appropriate relations indicated in Definition \ref{definition}
above. As a consequence, if $A$ is any $K$-algebra which contains a
set of elements satisfying these same relations (we call such a set
an $E$-\emph{family}), then there is a (unique) $K$-algebra
homomorphism from $L_K(E)$ to $A$ mapping the generators of $L_K(E)$
to their appropriate counterparts in $A$.  We will refer to this
conclusion as the {\it Universal Homomorphism Property} of $L_K(E)$.

If $F$ is a subgraph of $E$, then $F$ is called \emph{complete} in
case $s^{-1}_F(v)  = s^{-1}_E(v)$  for every $v\in F^0$ having
$s^{-1}_F(v) \neq \emptyset$.   In
particular, if $F$ is a complete subgraph of $E$ then the Universal Homomorphism Property of $L_K(F)$ yields that there is a
$K$-algebra homomorphism $L_K(F) \rightarrow L_K(E)$ mapping vertices and edges in $F$ with their counterparts in $E$. This homomorphism is in fact a $K$-algebra  {\it mono}morphism by \cite[Lemma 1.1]{AbAnhLouP}.

Many well-known algebras arise as the Leavitt path algebra of a
row-finite graph. For instance (see e.g. \cite[Examples 1.4]{AA1}),
the classical Leavitt algebras $L_n$ for $n\ge 2$ arise as the
algebras $L(R_n)$, where $R_n$ is the ``rose with $n$ petals" graph
$$
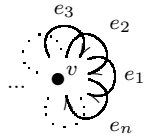
\xymatrix{ & {\bullet^v} \ar@(ur,dr) ^{e_1} \ar@(u,r) ^{e_2}
\ar@(ul,ur) ^{e_3} \ar@{.} @(l,u) \ar@{.} @(dr,dl) \ar@(r,d) ^{e_n}
\ar@{}[l] ^{\ldots} }$$ The full $n\times n$ matrix algebra over $K$
arises as the Leavitt path algebra of the oriented $n$-line graph
$$\xymatrix{{\bullet}^{v_1} \ar [r] ^{e_1} & {\bullet}^{v_2} \ar [r]
^{e_2} & {\bullet}^{v_3} \ar@{.}[r] & {\bullet}^{v_{n-1}} \ar [r]
^{e_{n-1}} & {\bullet}^{v_n}} $$ while the Laurent polynomial
algebra $K[x,x^{-1}]$ arises as the Leavitt path algebra of the
``one vertex, one loop" graph
$$
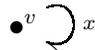
\xymatrix{{\bullet}^{v} \ar@(ur,dr) ^x}$$

\medskip
\noindent
 Constructions such as
direct sums and the formation of matrix rings produce additional
examples of Leavitt path algebras.

We recall now some information and establish notation for unital rings which will be used throughout Sections \ref{Morita Equiv} and \ref{Isomorphism}. We write
$$R \sim_M S$$
 to denote that $R$ is Morita equivalent to $S$.   For any ring $R$ we let ${\mathcal V}(R)$ denote the monoid of isomorphism classes of finitely generated projective left $R$-modules, with operation $\oplus$.  Since ${\mathcal V}(R)$ is conical (i.e., $[P]\oplus [Q] = [0]$ in ${\mathcal V}(R)$ if and only if $[P]=[Q]=[0]$), in fact
$${\mathcal V}^*(R)= {\mathcal V}(R) \setminus \{[0]\}$$
 is a semigroup as well.    For any graph $E$,  $\{[L(E)v] | v\in E^0\}$ is a set of generators for ${\mathcal V}^*(L(E))$.  If $\Phi:R{\rm -Mod} \rightarrow S{\rm -Mod}$ is a Morita equivalence, then the restriction
 $$\Phi_{{\mathcal V}}: {\mathcal V}^*(R) \rightarrow {\mathcal V}^*(S)$$
 is an isomorphism of semigroups.

 A nonzero idempotent $e$ in a ring $R$ is called {\it infinite} in case there
exist nonzero idempotents $f,g$ for which $e=f+g$, and $Re \cong Rf$ as left
  $R$-modules.  (That is, $e$ is infinite in case the left ideal $Re$ contains a proper direct summand isomorphic to itself.)  A simple unital ring $R$ is called {\it purely infinite} in case every nonzero left ideal of $R$ contains an infinite idempotent.

By \cite[Propositions 2.1 and 2.2]{AGP}, if $R$ is purely infinite simple, then ${\mathcal V}^*(R) = K_0(R)$ (the Grothendieck group of $R$).  In particular, any two elements of ${\mathcal V}^*(R)$ which are equal in $K_0(R)$ are in fact isomorphic as left $R$-modules.   Thus for $R,S$ Morita equivalent purely infinite simple rings, a Morita equivalence $\Phi: R{\rm -Mod} \rightarrow S{\rm -Mod}$ in fact restricts to an isomorphism
$$\Phi_{{\mathcal V}}: K_0(R) \rightarrow K_0(S).$$
We note that, in general, such an induced isomorphism of $K_0$ groups need not take $[1_R]$ to $[1_S]$.

Although $L(E)$ can be constructed for any graph $E$, the
Classification Question which is the main subject of this paper
pertains to those choices of $E$ for which $L(E)$ is purely infinite
simple unital.  It is easy to verify that $L(E)$ is unital if and only if $E^0$ is
finite  (in which case  $\sum _{v\in E^0} v=1_{L(E)}$), a fact that we will use throughout without explicit mention.    Thus
for much of the discussion we will assume that $E^0$ is finite; since for row-finite $E$ the finiteness of $E^0$ implies the finiteness of $E^1$, we simply write $E$ {\it is finite} in this case.    By
\cite[Theorem 3.11]{AA1} (and by substituting an equivalent
characterization from \cite[Lemma 2.8]{APS} for one of the
conditions therein), we get

\medskip

{\bf Simplicity Theorem.} For $E$ finite,  $L(E)$ is simple
precisely when every cycle of $E$ contains an exit, and there exists
a path in $E$ from any vertex to any cycle or sink.

\medskip

Furthermore, it is shown in \cite[Theorem 11]{AA2} that

\medskip

{\bf Purely Infinite Simplicity Theorem.}   $L(E)$ is purely
infinite simple precisely when $L(E)$ is simple, and $E$ contains a
cycle.

\medskip

Note that, as a consequence, whenever $L(E)$ is purely infinite
simple, $E$ does not contain sinks.

\section{Sufficient conditions for Morita equivalence between purely infinite simple unital Leavitt path algebras}\label{Morita Equiv}

In this section we establish sufficient conditions on two finite graphs
$E$ and $F$ which guarantee that $L(E)$ is Morita equivalent to
$L(F)$.  In the first step of this process, we build a cache of
operations on graphs that preserve Morita equivalence of the
associated Leavitt path algebras.  Once this arsenal is large
enough, the sufficiency result will follow from a well-known theorem
of Franks from symbolic dynamics, specifically, from the theory of
subshifts of finite type.  Our initial goal is to establish enough
such Morita-equivalence-preserving operations to allow us to apply
Franks' Theorem.  With that in mind, we prove only very restrictive
versions of the germane properties here, in order to significantly
streamline the proofs and arrive at our main results with maximum
haste.   (For instance, we present results here only for finite graphs, even though many of these results  hold for all row-finite graphs.) For completeness, we provide much more general versions of these properties in
Section \ref{Extra}.

 Our goal in this section is to establish a Morita equivalence result, i.e., a
 result which establishes the existence a Morita equivalence between various
 Leavitt path algebras.   However, a specific description of these equivalences,
 in particular a description of the restriction of these equivalences to the ${\mathcal V}^*$-semigroups, will be central to our discussion in the subsequent section; we
 therefore provide such additional information in Propositions \ref{sourceelimnationprop},  \ref{expansionprop}, \ref{insplittingprop}, and \ref{outsplittingprop}.

The key lemma which will be used to establish Morita equivalences throughout
this section is:

\begin{lem}\label{morita_lem}
Suppose $R$ and $S$ are  simple unital rings.  Let  $\pi : R
\rightarrow S$ be a nonzero,  not-necessarily-identity-preserving
ring homomorphism, and let $g$ denote the idempotent $\pi(1_R)$ of
$S$. If $g S g  =  \pi(R)$,  then there exists a Morita equivalence
$\Phi : R{\rm -Mod} \rightarrow S{\rm -Mod}$.

Moreover, $\Phi$ restricts to
an isomorphism $\Phi_{\mathcal{V}} : \mathcal{V}^*(R) \rightarrow
\mathcal{V}^*(S)$ with the property that for any idempotent $e \in R$,
$$\Phi_{\mathcal{V}}([Re]) = [S\pi(e)].$$
\end{lem}
\begin{proof}
That $\pi$ is nonzero, together with the simplicity of $R$, ensures
an isomorphism $R \cong \pi(R) = gSg$ as rings.  This gives a Morita
equivalence \[ \Pi : R{\rm -Mod} \to gSg{\rm -Mod}, \] given on objects by defining,
for each left $R$-module $M$, $ \Pi(M) = M^g $, where $M^g = M$ has $gSg$-action given by $gsg \ast m = \pi^{-1}(gsg)m.$

On the other hand, since $g = \pi(1_R) \neq 0$, the simplicity of $S$ ensures that
$SgS=S$, from which we conclude that the finitely generated projective left $S$-module
${}_SSg$ is a generator of the category of left $S$-modules.   Thus by the
well-known result of Morita we get a Morita equivalence \[ \Psi : gSg{\rm -Mod} \to S{\rm -Mod} \]
given by defining, for any $gSg$-module $N$, $\Psi(N) = Sg \otimes_{gSg} N.$

The composition of these two Morita equivalences gives a Morita equivalence \[ \Phi : R{\rm -Mod} \to S{\rm -Mod}. \]
Specifically, for each left $R$-module $M$,  $\Phi(M) = Sg \otimes_{gSg} M^g.$   In particular, $\Phi$
restricts to an isomorphism \[ \Phi_\mathcal{V} : \mathcal{V}^*(R) \to \mathcal{V}^*(S). \]

It is tedious and straightforward to show, for each $e=e^2\in R$, that $Sg \otimes_{gSg} (Re)^g \cong S\pi(e)$ as left $S$-modules, so the second statement follows as well.
\end{proof}

Suppose $E$ is a finite graph, let $X$ be any set of distinct vertices of $E$, and let $e=\sum_{v\in X}v \in L(E)$. It is immediate that every $y\in eL(E)e$ can be written as a $K$-linear combination of monomials of the form $\mu\nu^*$ for which $s(\mu), s(\nu) \in X$.   This observation will be used in the proofs of various results throughout the section without explicit mention.

We now establish the first of the four Morita equivalence results
required to achieve Theorem \ref{Franksinvariantsufficient}.

\begin{defi}\label{sourceelimdefinition}
{\rm Let $E = (E^0, E^1, r, s)$ be a directed graph with at least
two vertices, and let $v \in E^0$ be a source.  We form the {\em
source elimination graph} $E_{\backslash v}$ of $E$ as follows:
\begin{align*}
E_{\backslash v}^0   &= E^0 \backslash \{ v \} \\
E_{\backslash v}^1   &= E^1 \backslash s^{-1}(v) \\
s_{E_{\backslash v}} &= s |_{E_{\backslash v}^1} \\
r_{E_{\backslash v}} &= r |_{E_{\backslash v}^1}
\end{align*}
}
\end{defi}

\begin{exem}
{\rm Let $E$ be the graph:
$$\xymatrix{ \bullet \ar@/^/[r] & \bullet \ar@/^/[l] & \bullet^v \ar[l]}$$
Then the source elimination graph $E_{\backslash v}$ is
$$\xymatrix{ \bullet \ar@/^/[r] & \bullet \ar@/^/[l]}$$}
\end{exem}

It is easy to see that as long as the graph $E$ contains a cycle,
repeated source elimination can be used to convert $E$ into a graph
with no sources.

\begin{prop}\label{sourceelimnationprop}
Let $E$ be a finite graph containing at least two vertices such
that $L(E)$ is simple, and let $v \in E^0$ be a source.
Then $L(E_{\setminus v})$ is Morita equivalent to $L(E)$, via
a Morita equivalence
$$\Phi^{{\rm elim}}: L(E_{\setminus v}){\rm -Mod} \rightarrow
L(E){\rm -Mod}$$
 for which
$\Phi^{{\rm elim}}_\mathcal{V}([L(E_{\backslash v})w]) = [L(E)w]$ for all vertices
$w$ of $E_{\backslash v}$.
\end{prop}
\begin{proof}
We begin by noting that, as an easy application of the Simplicity
Theorem, $L(E)$ is simple and unital if and only if $L(E_{\backslash
v})$ is simple and unital.   (The hypothesis that $E$ contains at
least two vertices ensures that we are not creating an empty graph
by eliminating a single vertex.)

From the definition of $E_{\backslash v}$, it is clear that
$E_{\backslash v}$ is a complete subgraph of $E$.  Thus, the
$K$-algebra map defined by the rule
\begin{eqnarray*}
\pi : L(E_{\backslash v}) & \to     & L(E) \\
      w                   & \mapsto & w    \\
      e                   & \mapsto & e    \\
      e^\ast              & \mapsto & e^\ast
\end{eqnarray*}
for every $w \in E_{\backslash v}^0$ and every $e \in E_{\backslash
v}^1$, is a nonzero ring homomorphism.

We claim that $\pi(L(E_{\backslash v})) = {\pi(1_{L(E_{\backslash
v})})}\,L(E)\,{\pi(1_{L(E_{\backslash v})})}$.  Note that by
definition we have ${\pi(1_{L(E_{\backslash v})})} = \sum_{w \in
E^0, w \neq v} w$.  The inclusion $\pi(L(E_{\backslash v}))
\subseteq {\pi(1_{L(E_{\backslash
v})})}\,L(E)\,{\pi(1_{L(E_{\backslash v})})}$ is immediate. For the
other direction, it suffices to consider an arbitrary $\mu_1 \mu_2^\ast \in
{\pi(1_{L(E_{\backslash v})})}\,L(E)\,{\pi(1_{L(E_{\backslash
v})})}$. Then $\mu_1$ and $\mu_2$ are paths in $E$ such that neither has
$v$ for its source, and their ranges are equal.  But if neither has
$v$ for a source, then since $v$ is a source itself, neither path
can pass through $v$ at all.  Therefore, $\mu_1$ and $\mu_2$ are
also paths in $E_{\backslash v}$, such that $\pi(\mu_1 \mu_2^\ast) =
\mu_1 \mu_2^\ast$.  This completes the argument that
$\pi(L(E_{\backslash v})) = {\pi(1_{L(E_{\backslash
v})})}\,L(E)\,{\pi(1_{L(E_{\backslash v})})}$.

Applying Lemma \ref{morita_lem}, we conclude that $L(E_{\setminus v})$ is Morita
equivalent to $L(E)$, and that the Morita equivalence
restricts to an isomorphism between ${\mathcal V}^*(L(E_{\backslash v}))$ and ${\mathcal V}^*(L(E))$ that maps $[L(E_{\backslash v})w]$ to $[L(E)w]$ for each
vertex $w$ of $E_{\backslash v}$.
\end{proof}

\begin{corol}\label{sourceelimtogetMoritaequiv}
Let $E$ be a finite graph for which $L(E)$ is purely infinite simple. Then there exists a graph $E'$ which contains no sources,
with the property that $L(E)$ is Morita equivalent to $L(E')$ via
a Morita equivalence
$$\Phi^{{\rm ELIM}}: L(E'){\rm -Mod} \rightarrow L(E){\rm -Mod}$$
 for which  $\Phi^{{\rm ELIM}}_\mathcal{V}([L(E')w])
= [L(E)w]$ for all vertices $w$ of $E'$.
\end{corol}
\begin{proof}
By continually applying the source elimination procedure described in Definition \ref{sourceelimdefinition}, we produce from the finite graph $E$ a new graph $E'$ having no sources.  We must show that $E'$ is not the empty graph; that is, we must show that if the source elimination process eventually leads to a graph $F$ with one vertex, then that vertex is not a source.  By Proposition \ref{sourceelimnationprop}, at each stage of the source elimination process the graph produced in the new stage has Leavitt path algebra Morita equivalent to the Leavitt path algebra of the graph in the previous stage.      By \cite[Proposition 10]{AA2}, purely infinite simplicity is a Morita invariant.   Thus $L(F)$ is purely infinite simple.   But a graph $F$ with one vertex for which $L(F)$ is purely infinite simple must contain at least one loop at that vertex (e.g., by the Purely Infinite Simplicity Theorem), so that the vertex is not a source, thus completing the proof.
\end{proof}

We now build the second of the four indicated Morita equivalence results.

\begin{defis}
{\rm Let $E = (E^0, E^1, r, s)$ be a directed graph, and let $v
\in E^0$. Let $v^*$ and $f$ be symbols not in $E^0 \cup E^1$.   We
form the {\em expansion graph} $E_v$ from $E$ at $v$ as follows:
\begin{align*}
E_v^0      &= E^0 \cup \{ v^\ast \} \\
E_v^1      &= E^1 \cup \{ f \} \\
s_{E_v}(e) &= \left\{ \begin{array}{ll}
              v      & \textrm{ if $e = f$} \\
              v^\ast & \textrm{ if $s_E(e) = v$} \\
              s_E(e) & \textrm{ otherwise}
              \end{array}\right. \\
r_{E_v}(e) &= \left\{ \begin{array}{ll}
              v^\ast & \textrm{ if $e = f$} \\
              r_E(e) & \textrm{ otherwise}
              \end{array}\right. \\
\end{align*}
\noindent Conversely, if $E$ and $G$ are graphs, and there exists a
vertex $v$ of $E$ for which $E_v=G$, then $E$ is called a {\em
contraction} of $G$.}
\end{defis}

\begin{exem}
{\rm Let $E$ be the graph: $$
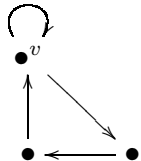
\xymatrix{\bullet^v \ar@(ul,ur) \ar[dr]
\\ \bullet \ar[u] & \bullet \ar[l] }$$ Then the expansion graph
$E_v$ is $$
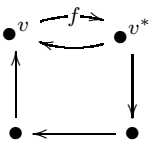
\xymatrix{\bullet^v \ar@/^/[r]|f & \bullet^{v^\ast}
\ar@/^/[l] \ar[d] \\ \bullet \ar[u] & \bullet \ar[l] }$$}
\end{exem}

\begin{prop}\label{expansionprop}
Let $E$ be a finite graph such that $L(E)$ is simple,
and let $v \in E^0$.  Then $L(E)$ is Morita equivalent to $L(E_v)$,
via a Morita equivalence
$$\Phi^{{\rm exp}}: L(E){\rm -Mod} \rightarrow L(E_v){\rm -Mod}$$
 for which $\Phi^{{\rm exp}}_\mathcal{V}([L(E)w])
= [L(E_v)w]$ for all vertices $w$ of $E$.
\end{prop}
\begin{proof}
We begin by noting that, as an easy application of the Simplicity
Theorem, $L(E)$ is simple and unital if and only if $L(E_{v})$ is
simple and unital.

For each $w \in E^0$, define $Q_w = w$.  For each $e \in s^{-1}(v)$,
define $T_e = fe$ and $T_e^\ast = e^\ast f^\ast$.  For $e \in E^1$
otherwise, define $T_e = e$ and $T_e^\ast$ = $e^\ast$.  We claim
that $\{ Q_w, T_e, T_e^\ast \mid w \in E^0, e \in E^1 \}$ is an
$E$-family in $L(E_v)$.  The $Q_w$'s are mutually orthogonal
idempotents because the $w$'s are.  The elements $T_e$ for $e \in
E^1$ clearly satisfy $T_e^\ast T_f = 0$ whenever $e \ne f$.  For $e
\in E^1$, it is easy to check that $T_e^\ast T_e = Q_{r(e)}$.  Note
that $\sum_{e \in s^{-1}(v)} T_e T_e^\ast = f \left(\sum_{e \in
s^{-1}(v^\ast)} e e^\ast \right) f^\ast = f f^\ast = v = Q_v$.  The
same property holds immediately for all $w \in E^0$ having $w \neq
v$, thereby establishing the claim.

Therefore, by the Universal Homomorphism Property of $L(E)$, there is a
$K$-algebra homomorphism $\pi : L(E) \to L(E_v)$ that maps $w \mapsto
Q_w$, $e \mapsto T_e$, and $e^\ast \mapsto T_e^\ast$.  Note that
$\pi$ maps $w$ to $Q_w \neq 0$, so $\pi$ is nonzero. We now claim
that $\pi(L(E)) = {\pi(1_{L(E)})}\,L(E_v)\,{\pi(1_{L(E)})}$, where
$\pi(1_{L(E)}) = \sum_{w \in E^0} w$, viewed as an element of $L(E_v)$.
The inclusion $\pi(L(E)) \subseteq {\pi(1_{L(E)})}\,L(E_v)\,{\pi(1_{L(E)})}$
is immediate.  For the other direction, it suffices to consider arbitrary nonzero
terms in ${\pi(1_{L(E)})}\,L(E_v)\,{\pi(1_{L(E)})}$ of the form $\mu_1 \mu_2^\ast$, where $\mu_1$ and $\mu_2$ are paths in $E_v$, $s(\mu_1), s(\mu_2) \neq v^\ast$, and
$r(\mu_1) = r(\mu_2)$.

Let $\alpha$ be the path in $E$ obtained by removing the edge $f$ from $\mu_1$
any place that it occurs, and similarly let $\beta$ be the path obtained by removing
$f$ from $\mu_2$.  We claim that $\pi(\alpha \beta^\ast) = \mu_1 \mu_2^\ast$.
There are two cases.  If $r(\mu_1) \neq v^\ast \neq r(\mu_2)$, then
$\mu_1 = \pi(\alpha)$ and $\mu_2 = \pi(\beta)$, and the result follows.
Otherwise, $r(\mu_1) = v^\ast = r(\mu_2)$.  But because $\mu_1$ and $\mu_2$ both
begin at a vertex other than $v^\ast$, and the only edge entering $v^\ast$ is $f$,
we must have $\mu_1 = \nu_1 f$ and $\mu_2$ = $\nu_2 f$, for paths $\nu_1, \nu_2$
in $E_v$, where $r(\nu_1) = v = r(\nu_2)$.  But then $\mu_1 \mu_2^\ast = \nu_1
f f^\ast \nu_2^\ast = \nu_1 \nu_2^\ast$ by the CK2 relation at $v$, and we are back
in the first case again, so $\pi(\alpha \beta^\ast) = \mu_1 \mu_2^\ast$, completing
the argument.

Applying Lemma \ref{morita_lem}, we conclude that $L(E)$ is Morita equivalent to $L(E_v)$,
and that the Morita equivalence restricts to the map given above.
\end{proof}

If $F$ is a contraction of $E$ (i.e., if there exists a vertex $v$ of $F$ for which $E=F_v$), then we denote by
$$\Phi^{{\rm cont}} = (\Phi^{{\rm exp}})^{-1}$$
the Morita equivalence from $L(F){\rm -Mod} \rightarrow L(E){\rm -Mod}.$
 We note that while the Morita equivalence $\Phi^{{\rm exp}}: L(E){\rm -Mod} \rightarrow L(E_v){\rm -Mod}$ arises from a ring homomorphism $\pi: L(E)\rightarrow L(E_v)$ as described in Lemma \ref{morita_lem}, the inverse equivalence $\Phi^{{\rm cont}}$ need not in general arise in this way.  For instance, if $E = \bullet^v$ is a graph with a single vertex $v$ and no edges, then $E_v = \bullet^v \rightarrow \bullet^{v^*}$, $L(E)\cong K$, $L(E_v)\cong {\rm M}_2(K)$, and $\pi: L(E)\rightarrow L(E_v)$ is the inclusion map to the upper left corner.   But there is no nonzero homomorphism from $L(E_v)$ to $L(E)$.

Our third and fourth Morita equivalence properties require somewhat more cumbersome
machinery to build than did the first two. The following
definition is borrowed from \cite[Section 5]{Flow}.

\begin{defis}\label{def_insplit}
{\rm
Let $E = ( E^0 , E^1 , r , s )$ be a directed graph. For each $v
\in E^0$ with $r^{-1}(v) \ne \emptyset$, partition the set $r^{-1}
(v)$ into disjoint nonempty subsets $\mathcal{E}^v_1 , \ldots ,
\mathcal{E}^v_{m(v)}$ where $m(v) \ge 1$. (If $v$ is a source then
we put $m(v)=0$.) Let $\mathcal{P}$ denote the resulting partition
of $E^1$. We form the {\em in-split graph} $E_r({\mathcal P})$
from $E$ using the partition $\mathcal{P}$ as follows:
\begin{align*}
E_r ( \mathcal {P} )^0 &= \{ v_i \mid v \in E^0 , 1 \le i \le m(v)
\}
\cup \{ v \mid m(v)  = 0 \} ,  \\
E_r ( \mathcal{P} )^1  &= \{ e_j \mid e \in E^1, 1 \le j \le m ( s
(e) ) \} \cup \{ e \mid m (s(e)) = 0 \} ,
\end{align*}

\noindent and define $r_{E_r ( \mathcal{P} )} , s_{E_r (
\mathcal{P} )} : E_r ( \mathcal{P} )^1 \rightarrow E_r (
\mathcal{P} )^0$ by
\begin{align*}
s_{E_r ( \mathcal{P} )} ( e_j ) &= s(e)_j  \text{ and } s_{E_r (
\mathcal{P} )} ( e
) = s(e) \\
r_{E_r ( \mathcal{P} )} ( e_j ) &= r(e)_i \text{ and } r_{E_r (
\mathcal{P} )} ( e ) = r(e)_i \text{ where } e \in
\mathcal{E}^{r(e)}_i .
\end{align*}
\noindent Conversely, if $E$ and $G$ are graphs, and there exists a
partition ${\mathcal P}$ of $E^1$ for which $E_r({\mathcal P})=G$,
then $E$ is called an {\em in-amalgamation} of $G$.}
\end{defis}

\begin{exem}
{\rm Let $E$ be the graph: $$
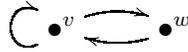
\xymatrix{\bullet^{v} \ar@(dl,ul)
\ar@/^/[r] & \bullet^{w} \ar@/^/[l]}$$ Denote by $\mathcal{P}$ the
partition of $E^1$ that places each edge in its own singleton
partition class.  Then $E_r(\mathcal{P})$ is:
$$
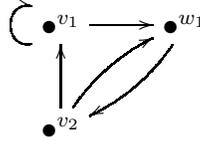
\xymatrix{\bullet^{v_1} \ar@(dl,ul) \ar[r] & \bullet^{w_1}
\ar@/^/[dl] \\ \bullet^{v_2} \ar[u] \ar@/^/[ur]}$$ }
\end{exem}

\begin{prop}\label{insplittingprop}
Let $E$ be a finite graph with no sources or sinks, such that $L(E)$ is simple.  Let $\mathcal{P}$ be a partition of $E^1$ as in Definitions \ref{def_insplit},
and $E_r(\mathcal{P})$ the in-split graph from $E$ using $\mathcal{P}$.
 Then $L(E)$ is Morita equivalent to $L(E_r(\mathcal{P}))$, via a Morita
 equivalence
 $$\Phi^{{\rm ins}}: L(E){\rm -Mod} \rightarrow L(E_r(\mathcal{P})){\rm -Mod}$$
  for which
$\Phi^{{\rm ins}}_\mathcal{V}([L(E)v]) = [L(E_r(\mathcal{P})) v_1]$ for all
vertices $v$ of $E$.
\end{prop}
\begin{proof}

We begin by noting that, as an easy application of the Simplicity
Theorem and a somewhat tedious check,  $L(E)$ is  simple and unital
if and only if $L(E_r(\mathcal{P}))$ is  simple and unital.
Moreover, $E$ has no sources if and only if $E_r(\mathcal{P})$ has
no sources.

For each $v \in E^0$, define $Q_v = v_1$, which exists by the
assumption that $E$ contains no sources.  For $e \in \mathcal{E}_i^v$,
define $T_e = \sum_{f \in s^{-1}(v)} e_1 f_i f_1^\ast$ and $T_e^\ast
= \sum_{f \in s^{-1}(v)} f_1 f_i^\ast e_1^\ast$.  The claim is that
$\{ Q_v, T_e, T_e^\ast \mid v \in E^0, e \in E^1 \}$ is an
$E$-family inside $L(E_r(\mathcal{P}))$.  The $Q_v$'s are mutually
orthogonal idempotents because the $v_1$'s are.  It is immediate from
the definition above that whenever $v = s(e)$ in $E$, then $Q_v T_e=
T_e$ and $T_e^\ast Q_v = T_e^\ast$ in $L(E_r(\mathcal{P}))$, and
that whenever $w = r(e)$ in $E$, $T_e Q_w = T_e$ and $Q_w T_e^\ast =
T_e^\ast$ in $L(E_r(\mathcal{P}))$.  If $e \neq f$, then note that
$T_e^\ast T_f = x e_1^\ast f_1 y$ for some $x, y \in
L(E_r(\mathcal{P}))$, but since $e_1 \neq f_1$, this is zero.  Because $E$
and $E_r(\mathcal{P})$ contain no sinks, there is a CK2 relation at every
vertex of both graphs.  It is now a straightforward matter of computation
to check, by applying the CK1 and CK2 relations, that $T_e^\ast T_e =
Q_{r(e)}$, and that $\sum_{e \in s^{-1}(v)} T_e T_e^\ast = Q_v$.

By the Universal Homomorphism Property, then, there exists a $K$-algebra homomorphism
$\pi:L(E)\rightarrow L(E_r(\mathcal{P}))$ which maps $v \mapsto
Q_v$, $e \mapsto T_e$, and $e^\ast \mapsto T_e^\ast$.  It is easy to
verify that $\pi(v)$ is nonzero for any $v \in E^0$, so $\pi$ is a
nonzero homomorphism. We now claim that $\pi(L(E)) =
{\pi(1_{L(E)})}\,L(E_r(\mathcal{P}))\,{\pi(1_{L(E)})}$, where
$\pi(1_{L(E)}) = \sum_{v \in E^0} v_1$.

The inclusion  $\pi(L(E))
\subseteq {\pi(1_{L(E)})}\,L(E_r(\mathcal{P}))\,{\pi(1_{L(E)})}$ is immediate.
For the opposite inclusion, it suffices to consider arbitrary nonzero terms in
${\pi(1_{L(E)})}\,L(E_r(\mathcal{P}))\,{\pi(1_{L(E)})}$ of the form   $\mu_1 \mu_2^\ast$, where $\mu_1$ and
$\mu_2$ are finite length paths in $E_r(\mathcal{P})$, and $s(\mu_1)
= v_1$ and $s(\mu_2) = w_1$ for some $v,w \in E^0$, and where
$r(\mu_1) = r(\mu_2)$.

Let $\mu$ be any path in $E_r(\mathcal{P})$ such that $s(\mu) = v_1$
for some $v \in E^0$.  Define $r(\mu) = w_k$, where $w \in E^0$ and
$1 \leq k \leq m(w)$.  We now build a path $\nu$ in $E$, by
replacing each $v_i$ in $\mu$ with $v$ in $\nu$, and each $e_i$ in
$\mu$ with $e$ in $\nu$, so that $\nu$ is essentially the result of
removing subscripts from the edges and vertices of $\mu$.  An
induction on the length of $\mu$ will show that
\[\pi(\nu) = \mu \left(\sum_{f \in s^{-1}(w)} f_k f_1^\ast \right).\]
If the length of $\mu$ is zero, then $\mu = v_1 = w_k$.  Applying
the CK2 relation at $v_1$, we get
\[\pi(v) = v_1 \left(\sum_{f \in s^{-1}(v)} f_1 f_1^\ast \right).\]
Since $w = v$ and $k = 1$ in this case, this is the result we need.
 If the length of $\mu$ is greater than zero, then $\mu = \mu' e_j$,
where $r(\mu') = u_j$, $e \in E^1$, $u \in E^0$, and $1 \leq j \leq
m(u)$.  We define $\nu'$ in the same manner as above, so that from
the inductive hypothesis,
\[\pi(\nu) = \pi(\nu') T_e = \mu' \left(\sum_{f \in s^{-1}(u),g \in s^{-1}(w)} f_j f_1^\ast e_1 g_k g_1^\ast\right).\]
When $f \neq e$, we have $f_1^\ast e_1 = 0$ by the CK1 relation,
whereas when $f = e$, $f_1^\ast e_1 = r(e_1)$, which collapses into
the adjacent terms.  This expression therefore reduces to
\[\pi(\nu) = \mu' e_j \left(\sum_{g \in s^{-1}(w)} g_k g_1^\ast \right) = \mu \left(\sum_{g \in s^{-1}(w)} g_k g_1^\ast \right)\]
as desired.

Now, given $\mu_1 \mu_2^\ast \in
{\pi(1_{L(E)})}\,L(E_r(\mathcal{P}))\,{\pi(1_{L(E)})}$, we define
$\nu_1$ and $\nu_2$ in the manner given above.  By a direct
computation, it can be verified that $\pi(\nu_1 \nu_2^\ast) = \mu_1
\mu_2^\ast$, completing the argument that $\pi(L(E)) =
{\pi(1_{L(E)})}\,L(E_r(\mathcal{P}))\,{\pi(1_{L(E)})}$.

Applying yet again Lemma \ref{morita_lem}, we conclude that $L(E)$
is Morita equivalent to $L(E_r(\mathcal{P}))$, and that the Morita
equivalence restricts to the map above.
\end{proof}

If $F$ is an in-amalgamation of $E$ (i.e., if there exists a vertex partition $\mathcal{P}$ of $F$ for which $E=F_r(\mathcal{P})$), then we denote by
$$\Phi^{{\rm inam}} = (\Phi^{{\rm ins}})^{-1}$$
the Morita equivalence from $L(F){\rm -Mod} \rightarrow L(E){\rm -Mod}.$

As a brief remark, we remind the reader that the result established
here is not as general as possible.  In particular, the hypothesis
that $E$ contains no sources or sinks will be weakened in Corollary
\ref{insplitme}.  (The difficulties that are avoided by this hypothesis
are more notational than substantial.  Nevertheless, the result as
stated here is strong enough to serve us for our present goal.)

We now establish the fourth and final tool in our cache.  The
following definition is borrowed from \cite[Section 3]{Flow}.

\begin{defis}\label{def_outsplit}
{\rm Let $E = ( E^0 , E^1 , r , s )$ be a directed graph. For each
$v \in E^0$ with $s^{-1}(v) \ne \emptyset$, partition the set
$s^{-1} (v)$ into disjoint nonempty subsets $\mathcal{E}^1_v ,
\ldots , \mathcal{E}^{m(v)}_v$ where $m(v) \ge 1$. (If $v$ is a sink,
then we put $m(v)=0$.) Let $\mathcal{P}$ denote the resulting
partition of $E^1$. We form the {\em out-split graph} $E_s({\mathcal
P})$ from $E$ using the partition $\mathcal{P}$ as follows:
\begin{align*}
E_s ( \mathcal {P} )^0 &= \{ v^i \mid v \in E^0 , 1 \le i \le m(v)
\}
\cup \{ v \mid m(v)  = 0 \} ,  \\
E_s ( \mathcal{P} )^1  &= \{ e^j \mid e \in E^1, 1 \le j \le m ( r
(e) ) \} \cup \{ e \mid m (r(e)) = 0 \} ,
\end{align*}

\noindent and define $r_{E_s ( \mathcal{P} )} , s_{E_s (
\mathcal{P} )} : E_s ( \mathcal{P} )^1 \rightarrow E_s (
\mathcal{P} )^0$ for each $e \in \mathcal{E}^i_{s(e)}$ by
\begin{align*}
s_{E_s ( \mathcal{P} )} ( e^j ) &= s(e)^i  \text{ and } s_{E_s (
\mathcal{P} )} ( e
) = s(e)^i \\
r_{E_s ( \mathcal{P} )} ( e^j ) &= r(e)^j \text{ and } r_{E_s (
\mathcal{P} )} ( e ) = r(e).
\end{align*}
\noindent Conversely, if $E$ and $G$ are graphs, and there exists a
partition ${\mathcal P}$ of $E^1$ for which $E_s({\mathcal P})=G$,
then $E$ is called an {\em out-amalgamation} of $G$. }
\end{defis}

\begin{exem}
{\rm Let $E$ be the graph: $$
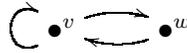
\xymatrix{\bullet^{v} \ar@(dl,ul)
\ar@/^/[r] & \bullet^{w} \ar@/^/[l]}$$ Denote by $\mathcal{P}$ the
partition of $E^1$ that places each edge in its own singleton
partition class.  Then $E_s(\mathcal{P})$ is:
$$
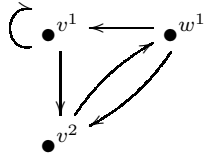
\xymatrix{\bullet^{v^1} \ar[d] \ar@(dl,ul) & \bullet^{w^1}
\ar@/^/[dl] \ar[l] \\ \bullet^{v^2} \ar@/^/[ur]}$$ }
\end{exem}

Our fourth tool follows as a specific case of a result previously
established in \cite{AbAnhLouP}.

\begin{prop}\label{outsplittingprop}
Let $E$ be a finite graph, $\mathcal{P}$ a partition of $E^1$ as
in Definitions \ref{def_outsplit}, and $E_s(\mathcal{P})$ the
out-split graph from $E$ using $\mathcal{P}$.  Then $L(E)$ is isomorphic to $L(E_s(\mathcal{P}))$.  This isomorphism
yields a Morita equivalence
$$\Phi^{{\rm outs}}: L(E){\rm -Mod} \rightarrow L(E_s(\mathcal{P})){\rm -Mod}$$
 for which
$\Phi^{{\rm outs}}_{{\mathcal V}}([L(E)v]) = [L(E_s(\mathcal{P}))\sum_{i=1}^{m(v)}v^i]$ for every vertex $v$ of $E$.
\end{prop}
\begin{proof}
The indicated isomorphism between $L(E)$ and $L(E_s(\mathcal{P}))$ is established in \cite[Theorem 2.8]{AbAnhLouP}.  Furthermore, the isomorphism
given there
maps $v$ to $\sum_{i=1}^{m(v)} v^i$, so that the associated Morita equivalence
restricts to the desired map.
\end{proof}

If $F$ is an out-amalgamation of $E$ (i.e., if there exists a vertex partition $\mathcal{P}$ of $F$ for which $E=F_s(\mathcal{P})$), then we denote by
$$\Phi^{{\rm outam}} = (\Phi^{{\rm outs}})^{-1}$$
the Morita equivalence from $L(F){\rm -Mod} \rightarrow L(E){\rm -Mod}.$

Having  built a sufficient arsenal of graph operations, we now
proceed toward the first main result of this article.  Considerable
work has been done in the flow dynamics community regarding  the
theory of subshifts of finite type; specifically, an explicit
description of the flow equivalence relation has been achieved for a
large class of such shifts.  We refer the interested reader to
\cite{L-M} for a clear, careful introduction to the theory,
including the definition of flow equivalence.  For our purposes, the
following definitions and results will provide all of the connecting
information we need.

\begin{defis}{\rm
Let $E$ be a finite (directed) graph.  Then $E$ is:

(1) \emph{irreducible} if given any two vertices $v$ and $w$ in $E$,
there is a path from $v$ to $w$  (\cite[Definition 2.2.13]{L-M}),

(2) \emph{essential} if there are neither sources nor sinks in $E$
(\cite[Definition 2.2.13]{L-M}), and

(3) \emph{trivial} if $E$ consists of a single cycle with no other
vertices or edges (\cite{Franks}). }
\end{defis}

Here is an easily verified observation which will be useful later.

\begin{lem}\label{keepcond}
Let $E$ be a finite graph, let $v\in E^0$, and let $\mathcal{P}$ be a partition of the vertices of $E$.
Then  $E$ is essential (resp. nontrivial, resp. irreducible) if and only if $E_s(\mathcal{P})$,
$E_r(\mathcal{P})$, and $E_v$  are each essential (resp. nontrivial, resp. irreducible).
\end{lem}

A set of graphs of great interest in the theory of subshifts of
finite type are those that are simultaneously irreducible,
essential, and nontrivial.  The following connecting result is
pivotal here.

\begin{lem}\label{graphconditions}
Let $E$ be a finite graph.  The following are equivalent:

(1) $E$ is irreducible, nontrivial, and essential.

(2) $E$ contains no sources, and $L(E)$ is purely infinite simple.
\end{lem}
\begin{proof}
Suppose first that $E$ is irreducible, essential, and nontrivial.  That
$E$ contains no sources is immediate.  $E$ also contains no sinks
and is finite, so it must contain a cycle.  Since $E$ is
nontrivial, there must exist some edge or vertex not in any cycle,
and either that edge or the path from the cycle to that vertex is an
exit to the cycle.  Finally, since $E$ is irreducible, there is a
path between any two vertices, so there must be a path from any
vertex to any cycle.

Conversely, suppose $E$ contains no sources, and that $L(E)$ is
purely infinite simple.  From the Simplicity Theorem \cite[Theorem
3.11]{AA1}, every cycle has an exit, so $E$ is nontrivial.  From
\cite[Lemma 2.8]{APS}, there is a path from any vertex to any cycle.
Since by the Purely Infinite Simplicity Theorem \cite[Theorem
11]{AA2} there is at least one cycle in the graph, there are no
sinks.  Then $E$ is essential.  However, with no sources or sinks in
a finite graph, every vertex must belong to a cycle, so there is a
path between any two vertices, and $E$ is irreducible.
\end{proof}

Much of the heavy lifting required to achieve our first goal is
provided by deep, fundamental work in flow dynamics.  We collect up
all the relevant facts in the following few results, then state as
Corollary \ref{flowsequenceforLPAs} the conclusion we need to achieve our
goal. (Following Franks, we state some results in the language of
matrices.  Statements about non-negative integer matrices may be
interchanged with statements about directed graphs by exchanging $E$
for its incidence matrix $A_E$ as described in the Introduction.)

\begin{defis}\label{standarddef}
{\rm We call a graph transformation {\it standard} if it is one of these six types: in-splitting, in-amalgamation, out-splitting, out-amalgamation, expansion, and contraction.  Analogously, we call
a function which transforms a non-negative integer matrix $A$ to a non-negative integer matrix $B$  {\it standard} if the corresponding
graph operation from $E_A$ to $E_B$ is standard.
}
\end{defis}

\begin{defis}\label{flowequivdef}
{\rm If $E$ and $F$ are graphs, a {\it flow equivalence from $E$ to $F$} is a sequence $E = E_0 \rightarrow E_1 \rightarrow \cdots \rightarrow E_n = F$ of graphs and standard graph transformations which starts at $E$ and ends at $F$.  We say that $E$ and $F$ are {\it flow equivalent} in case there is a flow equivalence from $E$ to $F$.    Analogously, a {\it flow equivalence} between matrices $A$ and $B$ is defined to be a flow equivalence between the graphs $E_A$ and $E_B$.}
\end{defis}

We note that the notion of flow equivalence can be described in topological terms (see e.g. \cite{L-M}). The definition given here is optimal for our purposes.  Specifically, it  agrees with the topologically-based definition for essential graphs by an application of \cite[Theorem]{P-S},  \cite[Corollary 4.4.1]{W}, and  \cite[Corollary 7.15]{L-M}.   Since all graphs under consideration in our main results are essential, this particular definition of flow equivalence will serve us most efficiently.

\begin{theor} \textrm{\cite[Theorem]{Franks}}\label{FranksTheorem} (``Franks' Theorem")
Suppose that $A$ and $B$ are non-negative irreducible square integer
matrices neither of which is in the trivial flow equivalence class.
Then the matrices $A$ and $B$ are flow equivalent if and only if:
\[
\det(I_n-A) = \det(I_m-B) \ \mbox{ and } \
\Z^n /(I_n-A)\Z^n \cong \Z^m /(I_m-B)\Z^m,
\]
where $n \times n$ and $m \times m$ are the sizes of $A$ and $B$ respectively, $I_n$
and $I_m$ are identity matrices, and $\Z^n /(I_n-A)\Z^n$ (resp.
$\Z^m /(I_m-B)\Z^m$) denotes the image in $\Z^n$ (resp. $\Z^m$) of
the linear transformation $I_n-A: \Z^n \rightarrow \Z^n$ (resp.
$I_m-B: \Z^m \rightarrow \Z^m$) induced by matrix multiplication.
\end{theor}

\begin{corol}\label{flowsequence}  Suppose $A$ and $B$ are irreducible,
nontrivial, essential square non-negative integer matrices for which
$$\det(I_n-A) = \det(I_m-B) \ \ \ \mbox{and} \ \ \ \Z^n /(I_n-A)\Z^n \cong \Z^m /(I_m-B)\Z^m.$$
Then there exists a sequence of standard transformations
which starts with $A$ and ends with $B$.
\end{corol}

As is usual, we denote $\Z^n /(I_n-A)\Z^n$ simply by ${\rm
coker}(I_n-A)$. By examining the Smith normal form of each matrix,
it is easy to show that ${\rm coker}(I_n-A) \cong {\rm
coker}(I_n-A^t)$ for any square matrix $A$.  Furthermore, by a
cofactor expansion, it is  clear that ${\rm det}(I_n-A)={\rm
det}(I_n-A^t) = {\rm det}(I_n-A)^t$.

If $E$ is a graph for which $L(E)$ is purely infinite simple unital,
then by \cite[Section 3]{AbAnhLouP} there is an isomorphism
$${\rm coker}(I-A_E^t) \rightarrow K_0(L(E)),$$
for which $\overline{x_i} \mapsto [L(E)v_i]$ for each standard basis
vector $x_i$ of $\Z^n$ and each vertex $v_i$ of $E$, $1\leq i \leq n$.   (We note for future use that since $1_{L(E)} = \sum_{v\in E^0}v$ in $L(E)$, this isomorphism takes the element $\sum_{i=1}^n \overline{x_i}$ of ${\rm coker}(I-A_E^t)$ to $[L(E)] \in K_0(L(E))$.)

Thus, using Lemma \ref{graphconditions}, we may
restate Corollary \ref{flowsequence} as follows:

\begin{corol}\label{flowsequenceforLPAs}
Suppose $G$ and $H$ are finite graphs without sources, for which $L(G)$ and
$L(H)$ are purely infinite simple. Suppose
$$\det(I_n-A_G^t) = \det(I_m-A_H^t) \ \ \ \mbox{and} \ \ \ K_0(L(G)) \cong K_0(L(H)),$$
where $n = | G^0 |$ and $m = | H^0 |$.  Then there exists a sequence
of standard graph transformations which starts with $G$
and ends with $H$.
\end{corol}

Our interest here will be in graphs $E$ and $F$ for which
$\det(I_n-A_E^t) = \det(I_m-A_F^t)$ and $K_0(L(E)) \cong K_0(L(F))$.
The following notation will prove convenient.

\begin{defi}{\rm Let $E$ be a finite graph.  The {\it determinant Franks pair}
is the ordered pair
$${\mathcal F}_{det}(E) = ( \ K_0(L(E)) \ , \ \det(I_n-A_E^t) \ )$$
consisting   of the abelian group $K_0(L(E))$ and the integer
$\det(I_n-A_E^t)$.  For finite graphs $E,G$ we write
$${\mathcal F}_{det}(E) \equiv {\mathcal F}_{det}(G)$$
in case there exists an abelian group isomorphism $K_0(L(E))\cong
K_0(L(G))$, and  $\det(I_n-A_E^t) = \det(I_m-A_G^t)$.  Clearly $\equiv$ yields an equivalence relation on the set of finite graphs. }
\end{defi}

We now show that the source elimination process preserves
equivalence of the determinant Franks pair.

\begin{lem}\label{samedet}
Let $E$ be a finite graph  for which $L(E)$ is
purely infinite simple, and let $v$ be a source in $E$.  Then
$${\mathcal F}_{det}(E) \equiv {\mathcal F}_{det}(E_{\backslash
v}).$$
\end{lem}
\begin{proof}
Let $n = |E^0|$.  Since $v$ is a source, $A_E$ contains a column of zeros. Then a
straightforward determinant computation by cofactors along this
column gives $\det(I_n-A_E^t) = \det(I_{n-1}-A_{E_{\backslash
v}}^t)$.

Since $E$ satisfies the conditions of the Purely Infinite Simplicity
Theorem it is clear by the construction that $E_{\backslash v}$ must
as well. But $L(E)$ and $L(E_{\backslash v})$ are Morita equivalent
by Proposition \ref{sourceelimnationprop}, so that their $K_0$
groups are necessarily isomorphic.
\end{proof}

Now we are ready to prove the first of our two main results.

\begin{theor}\label{Franksinvariantsufficient}
Let $E$ and $F$ be finite graphs such that $L(E)$ and $L(F)$ are
purely infinite simple.  Suppose that
$${\mathcal F}_{det}(E) \equiv {\mathcal F}_{det}(F);$$
that is, suppose
\[
\det(I_n-A_E^t) = \det(I_m-A_F^t) \ \ \ \mbox{and} \ \ \
K_0(L(E)) \cong K_0(L(F)),
\]
where $n$ and $m$ are the number of vertices in $E$ and $F$,
respectively.  Then $L(E)$ is Morita equivalent to $L(F)$.
\end{theor}
\begin{proof}
By Corollary \ref{sourceelimtogetMoritaequiv} there exist graphs
$E'$ and $F'$ such that $E'$ and $F'$ contain no sources, and for
which $L(E) \sim_M L(E')$ and $L(F) \sim_M L(F')$.  By hypothesis,
and by applying Lemma \ref{samedet} at each stage of the source
elimination process, we have that
$$\det(I-A_{E'}^t) = \det(I-A_E^t) = \det(I-A_F^t) = \det(I-A_{F'}^t),$$
and that
$$K_0(L(E'))\cong K_0(L(E)) \cong K_0(L(F)) \cong
K_0(L(F')).$$ Furthermore, $L(E')$ and $L(F')$ are each purely
infinite simple unital (either use the Purely Infinite Simplicity Theorem,
or apply the fact that purely infinite simplicity is a Morita
invariant).    So Corollary \ref{flowsequenceforLPAs} applies, and we
conclude that there exists a finite sequence of elementary graph transformations,  which starts at $E'$ and ends at $F'$.  By Lemmas
\ref{keepcond} and \ref{graphconditions}, since $E'$ is purely infinite simple unital with no
sources, each time such an operation is applied the resulting graph
is again purely infinite simple unital with no sources.      Thus,
at each step of the sequence, we may apply the appropriate tool from
the cache consisting of Propositions \ref{expansionprop},
\ref{insplittingprop}, and \ref{outsplittingprop}, from which we
conclude that each step in the sequence preserves Morita equivalence
of the corresponding Leavitt path algebras. Combining these Morita
equivalences at each step then yields $L(E') \sim_M L(F')$.

As a result, we have
$$
L(E) \sim_M L(E') \sim_M L(F') \sim_M L(F),$$
and the theorem follows.
\end{proof}

An analysis of objects related to Leavitt path algebras, carried out along somewhat similar lines, is presented in \cite{CK}.  To wit, Cuntz and Krieger analyze the C$^*$-algebras $\mathcal{O}_A$ (the now-so-called {\it Cuntz-Krieger C$^*$-algebras}); these are C$^*$-algebras generated by partial isometries which satisfy relations analogous to those in Definition \ref{definition}.  The arguments utilized in the C$^*$-algebra context are based on topological and analytical properties of  Markov chains and  certain graph operations (e.g., those of \cite{P-S}).  The completely algebraic point of view used in this section, together with the specific construction presented in \cite{Franks}, infuses our approach with a more germane algebraic flavor.

\section{Sufficient conditions for isomorphisms between purely infinite
simple unital Leavitt path algebras.}\label{Isomorphism}

In this section we will use the techniques and results of the
previous section to investigate the problem of classifying   purely
infinite simple unital Leavitt path algebras up to isomorphism.
Specifically, in Corollary \ref{Th:TorsionfreeBigFish} we provide an
affirmative answer to the Classification Question for a wide class
of graphs. To help establish such a connection we introduce some
notation.

\begin{defi}\label{unitaryFrankspair}{\rm Let $E$ be a finite graph.
The {\it unitary Franks pair} is the ordered pair
$${\mathcal F}_{[1]}(E) = ( \ K_0(L(E)) \ , \ [1_{L(E)}] \ )$$
consisting of the abelian group $K_0(L(E))$ and the element
$[1_{L(E)}]$ of $K_0(L(E))$. For finite graphs $E,G$ we write
$${\mathcal F}_{[1]}(E) \equiv {\mathcal F}_{[1]}(G)$$
in case there exists an abelian group isomorphism $\varphi:
K_0(L(E))\rightarrow K_0(L(G))$ for which
$\varphi([1_{L(E)}])=[1_{L(G)}]$.    Clearly $\equiv$ yields an equivalence
relation on the set of finite graphs. }
\end{defi}

We will show that, in the case of Morita equivalent purely infinite simple
Leavitt path algebras over finite graphs, if the unitary Franks pair of
their graphs are equivalent, then the algebras are isomorphic.  The argument
relies on the adaptation to our context of the deep result of Huang \cite[Theorem 1.1]{Huang}.

Now suppose $E$ has $L(E)$ purely infinite simple unital, and has no sources.  Then by Lemma \ref{graphconditions} $E^t$ has these same properties.   Let
$$E^t = H_0 \rightarrow^{m_1} H_1 \rightarrow^{m_2} H_2 ... \rightarrow^{m_n} H_n = E^t$$
be a finite sequence of standard graph transformations which starts and ends with $E^t$.   We write $H_i = G_i^t$ (where $G_i = H_i^t$), and so we have a finite sequence of graph transformations
$$E^t = G_0^t \rightarrow^{m_1} G_1^t \rightarrow^{m_2} G_2^t ... \rightarrow^{m_n} G_n^t = E^t.$$

For any graph $G$ let $\tau_G: G \rightarrow G^t$ be the graph function which is the identity on vertices, but switches the direction of each of the edges. (This is simply the transpose operation on the corresponding incidence matrices.)   In particular, any one of the standard graph transformations
$$m: G^t_i \rightarrow G^t_{i+1}$$
 yields a graph transformation
 $$m' = \tau_{G_{i+1}}^{-1}\circ m \circ \tau_{G_i} : G_i \rightarrow G_{i+1}.$$

 \begin{lem}\label{standardlemma}
 If $m: G_i^t\rightarrow G_{i+1}^t$ is a standard  graph transformation, then $m'=\tau_{G_{i+1}}^{-1}\circ m \circ \tau_{G_i} : G_i \rightarrow G_{i+1}$ is also standard.
  \end{lem}

  \begin{proof}
  We leave to the reader the straightforward check that
  \begin{enumerate}
  \item If $m$ is an expansion (resp. contraction), then $m'$ is an expansion (resp. contraction).
  \item If $m$ is an in-splitting (resp. out-splitting), then $m'$ is an out-splitting (resp. in-splitting).
  \item  If $m$ is an in-amalgamation (resp. out-amalgamation), then $m'$ is an out-amalgamation (resp. in-amalgamation).
  \end{enumerate}
  (We note that expansions (resp. contractions) remain expansions (resp. contractions) when passing to the transpose graph, but that the other four standard operations indeed become a different type of standard transformation on the transpose.)
  \end{proof}

  As a consequence of  Lemma \ref{standardlemma}, if we start with any finite sequence of standard graph transformations
$$E^t = H_0 \rightarrow^{m_1} H_1 \rightarrow^{m_2} H_2 ... \rightarrow^{m_n} H_n = E^t$$
which starts and ends with $E^t$,  then we get a corresponding finite sequence of standard graph transformations
$$E = G_0 \rightarrow^{m'_1} G_1 \rightarrow^{m'_2} G_2 ... \rightarrow^{m'_n} G_n = E$$
which starts and ends with $E$.

  By \cite[Lemma 3.7]{Huang}, for any graphs $E$ and $F$, any standard graph transformation $m:E\rightarrow F$ yields the so-called {\it induced} isomorphism
  $$\varphi_m: coker(I-A_{E}) \rightarrow coker(I-A_{F}).$$
  For each of the six types of standard graph transformations, the corresponding induced isomorphism is explicitly described in \cite[Lemma 3.7]{Huang}.  As a representative example of these induced isomorphisms, we offer the following description.   Suppose $m: E \rightarrow F$ is an in-splitting; that is,  $F = E_r(\mathcal{P})$ for some partition ${\mathcal P}$ of the edges of $E$.   By generalizing the construction of
Franks \cite[Theorem 1.7]{Franks} in the natural way, we define matrices
\[
R = \left(\begin{array}{cccc}
a_{11} & a_{12} & \cdots & a_{1n} \\
a_{21} & a_{22} & \cdots & a_{2n} \\
\vdots & \vdots & \ddots & \vdots \\
a_{m1} & a_{m2} & \cdots & a_{mn} \\
\end{array}\right)
\qquad
S = \left(\begin{array}{cccc}
1      & 0      & \cdots & 0      \\
\vdots & \vdots & \ddots & \vdots \\
1      & 0      & \cdots & 0      \\
0      & 1      & \cdots & 0      \\
\vdots & \vdots & \ddots & \vdots \\
0      & 1      & \cdots & 0      \\
\vdots & \vdots & \ddots & \vdots \\
0      & 0      & \cdots & 1      \\
\vdots & \vdots & \ddots & \vdots \\
0      & 0      & \cdots & 1      \\
\end{array}\right)
\]
where $a_{ij}$ is the number of edges in the $j$th partition class of $E^1$ that leave vertex $i$.  The
columns of $S$ correspond to vertices of $E$, and the rows to partition classes of $E^1$, where a $1$
indicates that a partition class contains edges entering the vertex.  With $R$ and $S$ so defined, it is straightforward to show that  $A_E=RS$ and $A_F=SR$.  By
\cite[Lemma 3.7]{Huang}, we get $[x] \mapsto [Rx]$ is the induced isomorphism on
${\rm coker}(I_n-A_E)$.

  As it turns out, the descriptions of the induced isomorphisms coming from expansions and contractions are somewhat different than  the descriptions of the induced isomorphisms coming from the other four types of standard graph transformations.   However, in each case, an explicit description of the induced isomorphism can be given as above.   We leave the details in the other five cases to the interested reader.   Here now is the connection between the Morita equivalences of Section \ref{Morita Equiv} and the induced isomorphisms given by Huang.

  \begin{prop}\label{samerestrictionprop}
  Let $G_i$ and $G_{i+1}$ be finite graphs.   Suppose $G_i$ has $L(G_i)$ purely infinite simple, and has no sources.  Suppose $m_i:G_i^t \rightarrow G_{i+1}^t$ is a standard graph transformation, and let $\varphi_{m_i}: coker(I-A_{G_i^t})\rightarrow coker(I-A_{G_{i+1}^t})$ be the induced isomorphism.    Let $m'_i: G_i \rightarrow G_{i+1}$ be the corresponding graph transformation, which, by Lemma \ref{standardlemma}, is also a standard transformation.   Let $\Phi^{m'_i}:L(G_i){\rm -Mod} \rightarrow L(G_{i+1}){\rm -Mod}$ be the Morita equivalence induced by $m'_i$ as described in Section \ref{Morita Equiv}.   Then, using the previously described identification between $K_0(L(G_i))$ and $coker(I-A^t_{G_i})$ (resp. between $K_0(L(G_{i+1}))$ and $coker(I-A^t_{G_{i+1}})$), we have
  $$\Phi^{m_i'}_{\mathcal{V}} = \varphi_{m_i}.$$

  \end{prop}

\begin{proof}
Each of the six types of isomomorphisms $\Phi^{m_i'}_{\mathcal{V}}: K_0(L(G_i))\rightarrow K_0(L(G_{i+1}))$ have been explicitly described in Section \ref{Morita Equiv}.  As indicated above, each of the six types of induced isomorphisms $\varphi_{m_i}: coker(I-A_{G_i^t})\rightarrow coker(I-A_{G_{i+1}^t})$ have been explicitly described in \cite[Lemma 3.7]{Huang}.  By definition we have $A_{G_i^t} = A^t_{G_i}$ (resp. $A_{G^t_{i+1}} = A^t_{G_{i+1}})$.   It is now a tedious but completely straightforward check to verify that, in all six cases, these isomorphisms agree.
\end{proof}

We are finally in position to adapt the result of Huang to our context. For a ring $R$, and an automorphism $\alpha$ of $K_0(R)$, we say a Morita equivalence $\Phi: R{\rm -Mod} \rightarrow R{\rm -Mod}$ {\it restricts to} $\alpha$ in case $\Phi_{\mathcal{V}} = \alpha.$

\begin{prop}\label{huangwithmorita}
Let $E$ be a finite graph for which  $L(E)$ is purely infinite simple.   Let $\alpha$ be any automorphism of $K_0(L(E))$.   Then there exists a Morita equivalence $\Phi : L(E){\rm -Mod} \to L(E){\rm -Mod}$ which restricts to $\alpha$.
\end{prop}
\begin{proof}
If $E$ contains sources, then Corollary \ref{sourceelimtogetMoritaequiv} guarantees the existence of a Morita equivalence
$\Phi^{{\rm ELIM}}: L(E'){\rm -Mod} \rightarrow L(E){\rm -Mod}$, where $E'$ has no sources.  If $\Psi: L(E'){\rm -Mod} \rightarrow L(E'){\rm -Mod}$ is a Morita equivalence which restricts to the automorphism $(\Phi^{{\rm ELIM}}_{\mathcal{V}})^{-1} \circ \alpha \circ \Phi^{{\rm ELIM}}_{\mathcal{V}}$ of $K_0(L(E'))$, then $\Phi^{{\rm ELIM}} \circ \Psi \circ (\Phi^{{\rm ELIM}})^{-1}$ is a Morita equivalence from $L(E){\rm -Mod}$ to $L(E){\rm -Mod}$ which restricts to $\alpha$.   Therefore, it suffices to consider graphs $E$ with no sources.

If $L(E)$ is purely infinite simple, and $E$ has no sources, then $E$ is essential, irreducible, and
non-trivial, and hence so is $E^t$. Since $K_0(L(E))$ is identified with $coker(I - A_E^t)$, we may view $\alpha$ as an automorphism of $coker(I-A_E^t)=coker(I-A_{E^t})$.  Therefore, by \cite[Theorem 1.1]{Huang} (details in \cite[Theorem 2.15]{Huang2}), there exists
a flow equivalence ${\mathcal F}$ from $E^t$ to itself which induces $\alpha$.   Such a flow equivalence can be written as a finite sequence
$$E^t = H_0 \rightarrow^{m_1} H_1 \rightarrow^{m_2} H_2 ... \rightarrow^{m_n} H_n = E^t$$
of standard graph transformations which starts and ends with $E^t$.  But this then yields  a corresponding finite sequence of standard graph transformations
$$E = G_0 \rightarrow^{m'_1} G_1 \rightarrow^{m'_2} G_2 ... \rightarrow^{m'_n} G_n = E$$
which starts and ends with $E$, as described in Lemma \ref{standardlemma}.
This sequence of standard graph transformations in turn yields a sequence of Morita equivalences (using the results of Section \ref{Morita Equiv}) which starts and ends at $L(E){\rm -Mod}$.   But by Proposition \ref{samerestrictionprop},    at each stage of the sequence the restriction of the Morita equivalence to the appropriate $K_0$ group agrees with the induced map coming from the standard graph transformation.   If we denote by $\Phi:L(E){\rm -Mod} \rightarrow L(E){\rm -Mod}$ the composition of these Morita equivalences, then $\Phi$ restricts to the same automorphism of $K_0(L(E))$ as does ${\mathcal F}$, namely, the prescribed automorphism $\alpha.$
\end{proof}

Here now is the second main result of this article.

\begin{theor}\label{Th:morita_iso}
Let $E,G$ be finite graphs such that $L(E),L(G)$ are purely
infinite simple unital Leavitt path algebras, and such that $L(E)$ is Morita equivalent to $L(G)$. If
$$\mathcal{F}_{[1]}(L(E)) \equiv \mathcal{F}_{[1]}(L(G)),$$
(i.e., if $K_0(L(E)) \cong K_0(L(G))$ via an isomorphism which sends
$[1_{L(E)}]$ to $[1_{L(G)}]$), then there is a ring isomorphism $$L(E)\cong L(G).$$
\end{theor}
\begin{proof}
Suppose that $\varphi :K_0(L(E))\rightarrow K_0(L(G))$ is an
isomorphism with $\varphi ([1_{L(E)}])=[1_{L(G)}]$.  Since
$L(E)$ and $L(G)$ are Morita equivalent by hypothesis, there exists
a Morita equivalence $$\Gamma : L(E){\rm -Mod}\rightarrow L(G){\rm -Mod}.$$
Thus there is an isomorphism $\Gamma_{\mathcal{V}}:
K_0(L(E))\rightarrow K_0(L(G))$.

Now consider the group automorphism $$\varphi\circ \Gamma_{\mathcal{V}}^{-1}:
K_0(L(G))\rightarrow K_0(L(G)).$$
By Proposition \ref{huangwithmorita}, there
exists a Morita equivalence $\Psi : L(G){\rm -Mod}\rightarrow L(G){\rm -Mod}$ such that
$$\Psi_{\mathcal{V}}=\varphi\circ \Gamma_{\mathcal{V}}^{-1}.$$
Thus, we get a Morita equivalence
$$H:=\Psi\circ \Gamma : L(E){\rm -Mod}\rightarrow L(G){\rm -Mod}$$
with
$$H_{\mathcal{V}}=(\Psi\circ \Gamma)_{\mathcal{V}}=\Psi_{\mathcal{V}}\circ \Gamma_{\mathcal{V}}=
\varphi\circ \Gamma_{\mathcal{V}}^{-1}\circ \Gamma_{\mathcal{V}}=\varphi .$$
 In particular,
$H_{\mathcal{V}}([1_{L(E)}])=\varphi ([1_{L(E)}])=[1_{L(G)}]$. As noted in
the Introduction, since
$L(E)$ and $L(G)$ are purely infinite simple rings, \cite[Corollary
2.2]{AGP} implies that $[1_{L(E)}]\in K_0(L(E))$ consists of the
finitely generated projective left $L(E)$-modules isomorphic (as
left $L(E)$-modules) to the progenerator ${}_{L(E)}L(E)$, and
analogously $[1_{L(G)}]\in K_0(L(G))$ consists of the finitely
generated projective left $L(G)$-modules isomorphic (as left
$L(G)$-modules) to the progenerator ${}_{L(G)}L(G)$. Thus the
equation $H_{\mathcal{V}}([1_{L(E)}])=[1_{L(G)}]$ yields that
$H({}_{L(E)}L(E))\cong {}_{L(G)}L(G)$. Since Morita equivalences
preserve endomorphism rings, we get ring isomorphisms
$$L(E)\cong \mbox{End}_{L(E)}(L(E))\cong
\mbox{End}_{L(G)}(H(L(E)))
\cong \mbox{End}_{L(G)}(L(G))\cong L(G) ,$$ and the theorem is established.
\end{proof}

An easy corollary now gives sufficient, readily computable, and remarkably weak
conditions under which two unital purely infinite simple Leavitt path
algebras are known to be isomorphic.  We combine $\mathcal{F}_{det}$ and
$\mathcal{F}_{[1]}$ to obtain these conditions.

\begin{defi}\label{iso-invariant}
{\rm Let $E$ be a finite graph. We define the {\it Franks
triple}  to be the ordered triple
$${\mathcal F}_3(E)=( \ K_0(L(E)),[1_{L(E)}] , \mbox{det}(I_n-A_E^t) \ ),$$ consisting of the abelian group
$K_0(L(E))$, the element $[1_{L(E)}]$  which represents the
order unit of
$K_0(L(E))$ containing $1_{L(E)}$, and the integer $\det(I_n-A_E^t)$ (where $n= | E^0 |$).
For finite graphs $E,G$ we write
$${\mathcal F}_{3}(E) \equiv {\mathcal F}_{3}(G)$$
in case there exists an abelian group isomorphism $\varphi:
K_0(L(E))\rightarrow K_0(L(G))$ for which $\varphi([1_{L(E)}]) =
[1_{L(G)}]$, and  $\det(I_n-A_E^t) = \det(I_m-A_G^t)$.
Clearly $\equiv$ yields an equivalence relation on the set of finite graphs. }
\end{defi}

When $n=|E^0|$ is clear from context we will often denote the $n\times n$ identity
matrix $I_n$ simply by $I$.  We now have the sufficiency result we pursued.

\begin{corol}\label{Th:Almost BigFish}
Let $E,G$ be finite graphs such that $L(E)$ and $L(G)$ are purely
infinite
simple Leavitt path algebras. If
$$\mathcal{F}_3(L(E))\equiv \mathcal{F}_3(L(G)),$$
(i.e., if $K_0(L(E))\cong K_0(L(G))$ via an isomorphism which sends
$[1_{L(E)}]$ to $[1_{L(G)}]$, and
$\mbox{det}(I-A_E^t)=\mbox{det}(I-A_G^t)$), then there is a ring isomorphism
$$L(E)\cong L(G).$$
\end{corol}
\begin{proof}
Since $\mathcal{F}_3(L(E))\equiv \mathcal{F}_3(L(G))$, we have in particular
that $\mathcal{F}_{det}(L(E))\equiv \mathcal{F}_{det}(L(G))$, so that
$L(E)$ and $L(G)$ are Morita equivalent by Theorem
\ref{Franksinvariantsufficient}.  At the same time, we have
$\mathcal{F}_{[1]}(L(E))\equiv \mathcal{F}_{[1]}(L(G))$, which together
with Theorem \ref{Th:morita_iso} gives the isomorphism we seek.
\end{proof}

\begin{exem}
{\rm

Let $E$ and $F$ be the graphs

\medskip

 \[
E = R_4 = \qquad 
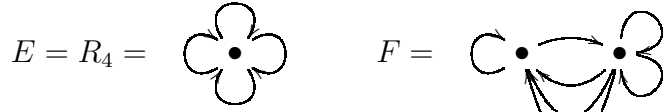
\xymatrix{\bullet \ar@(ul,ur) \ar@(ur,dr) \ar@(dr,dl) \ar@(dl,ul)} \qquad \qquad
F = \qquad \xymatrix{\bullet \ar@(dl,ul) \ar@/^/[r] & \bullet \ar@/^1pc/[l] \ar@/^2pc/[l]\ar@/^3pc/[l] \ar@(u,r) \ar@(r,d)}
\]

\medskip

Then
$$A_E = (4) \ \ \mbox{and} \ \ A_F = \left(
\begin{array}{cc}
1  & 1   \\
 3 &  2  \\
\end{array}
\right),$$
so that
$$I - A^t_E = (-3) \ \ \mbox{and} \ \ I - A^t_F = \left(
\begin{array}{cc}
0  & -3   \\
 -1 & -1  \\
\end{array}
\right).$$
It is well known (and easy to compute) that $K_0(L(E))\cong \Z_3$, with $[1_{L(E)}] = 1$ in $\Z_3$.   Similarly, it is not hard to show that $K_0(L(F)) \cong coker(I - A_F^t) \cong \Z_3$ as well, and that $[1_{L(F)}]=1$ in $\Z_3$.   Clearly $det(I-A_E^t)=-3=det(I-A_F^t)$.   Since both graphs yield purely infinite simple unital Leavitt path algebras, we conclude by Corollary \ref{Th:Almost BigFish} that the Leavitt path algebras $L(E)$ and $L(F)$ are isomorphic.

}
\end{exem}

It is worth remarking that the proof of Proposition \ref{huangwithmorita}, and therefore the proof of Theorem \ref{Th:morita_iso}, hinges on results of Huang (\cite{Huang2} and  \cite{Huang}).  Huang's results  are {\it constructive}; specifically, the flow equivalence from $E^t$ to itself which induces $\alpha$ (as in the proof of Proposition \ref{huangwithmorita}) is explicitly described.  In \cite{AS} we present in detail an algorithmic description of the resulting isomorphism $L(E)\cong L(G)$ of Theorem \ref{Th:morita_iso}.

\medskip

Corollary \ref{Th:Almost BigFish} establishes that equivalence of the
Franks triple is a sufficient condition to conclude isomorphism of
the corresponding purely infinite simple Leavitt path algebras over
finite graphs.    For the remainder of this section we consider
whether or not the unitary Franks pair (i.e., the pair
$(K_0(L(E)),[1_{L(E)}])$ {\it without} the ${\rm det}(I-A_E^t)$
information) precisely classifies these algebras. It is known that
the converse is true: namely,   that an isomorphism $L(E)\cong L(G)$
implies the equivalence  of the unitary Franks pairs ${\mathcal
F}_{[1]}(E) \equiv {\mathcal F}_{[1]}(G)$ (see e.g. \cite[Theorem
5.11]{AbAnhLouP}).

It turns out that equivalence of the unitary Franks invariants {\it almost}
guarantees equivalence of the corresponding Franks triples; the only possible
difference can be in the sign of the determinant.  In particular, we can
recast Corollary \ref{Th:Almost BigFish} as follows.

\begin{corol}\label{determinantabsval}
If $E$ and $G$ are finite graphs for which the Leavitt path algebras $L(E)$ and $L(G)$
are purely infinite simple, for which ${\mathcal F}_{[1]}(E) \equiv {\mathcal F}_{[1]}(G)$,
and for which the integers ${\rm det}(I-A_E^t)$ and ${\rm det}(I-A_G^t)$ have the same
sign, then there is a ring isomorphism $L(E)\cong L(G)$.
\end{corol}

\begin{proof}
Since ${\mathcal F}_{[1]}(E) \equiv {\mathcal F}_{[1]}(G)$, we have in particular that
$coker(I - A_E^t) \cong coker(I-A_G^t)$, whence the Smith normal forms of these two
matrices are the same.  But the Smith normal form of a matrix is achieved by a process
which involves multiplication by various matrices, each having determinant $1$ or $-1$.
In particular, this yields $|{\rm det}(I-A_E^t)| = |{\rm det}(I-A_G^t)|$.   So
${\rm det}(I-A_E^t)$ and ${\rm det}(I-A_G^t)$ having the same sign implies equality of
these two integers, whence the result follows from Corollary \ref{Th:Almost
BigFish}.
\end{proof}

We note that there are  classes of graphs for which  equivalence of the unitary Franks
pair automatically implies  equivalence of the corresponding Franks triple, which then in
turn implies  isomorphism of the corresponding Leavitt path algebras by Corollary
\ref{Th:Almost BigFish}.   For instance, isomorphisms between various sized matrix rings over
the Leavitt algebra $L_n$ (see the Introduction) can be recast as isomorphisms between
Leavitt path algebras over appropriate graphs  (see \cite[Section 5]{AbAnhLouP}).  In this
context, one can show that graphs having equivalent unitary Franks pair indeed have identical
(negative) ${\rm det}(I-A^t)$, so that \cite[Theorem 5.9]{AbAnhLouP} in fact follows from
Corollary \ref{Th:Almost BigFish}.

Similarly, isomorphisms between purely infinite simple Leavitt path algebras $L(E)$ and
$L(G)$, for which neither $E$ nor $G$ have parallel edges, and for which both
$|E^0| \leq 3$ and $|G^0| \leq 3$, are established in \cite[Section 4]{AbAnhLouP}.   In
this context as well, one can show that  graphs having equivalent unitary Franks pair
indeed have identical (negative) ${\rm det}(I-A^t)$, so that \cite[Propositions 4.1 and
4.2]{AbAnhLouP}  follow from Corollary \ref{Th:Almost BigFish} as well.

The previous two paragraphs notwithstanding, the cited isomorphism results from \cite{AbAnhLouP}
are more than merely special cases of Corollary \ref{Th:Almost BigFish}, since the isomorphisms
of \cite{AbAnhLouP} are in fact explicitly constructed.

An immediate, interesting consequence of Corollary \ref{Th:Almost
BigFish} is the following result along these same lines.

\begin{corol}\label{Th:TorsionfreeBigFish}
Let $E,G$ be finite graphs such that $L(E),L(G)$ are purely infinite
simple Leavitt path algebras with infinite Grothendieck groups. If
${\mathcal F}_{[1]}(E) \equiv {\mathcal F}_{[1]}(G)$, then
$L(E)\cong L(G)$. In other words, in this situation, equivalence of
the unitary Franks pairs is sufficient to yield isomorphism of the
Leavitt path algebras.
\end{corol}
\begin{proof}
The condition that $L(E)$ and $L(G)$ have infinite Grothendieck
groups implies that ${\rm det}(I-A_E^t) = 0 = {\rm det}(I-A_G^t)$,
and Corollary \ref{Th:Almost BigFish} then applies.
\end{proof}

So, in the case of infinite $K_0$-groups, the unitary Franks pair
$(K_0(L(E)), [1_{L(E)}])$ is an invariant for classifying the purely
infinite simple unital Leavitt path algebras  up to isomorphism.

We note that somewhat stronger conclusions may be drawn in Theorem \ref{Th:morita_iso} and the subsequent three Corollaries, to wit, that there exist $K$-algebra isomorphisms between the corresponding Leavitt path algebras.  The tools required to establish the existence of such $K$-algebra isomorphisms are more extensive than the tools utilized here.  This approach follows an approach similar to one developed by Cuntz, which is described in \cite[Theorem 6.5]{Rordam}.

For the remainder of this section we investigate whether or not the result
of Corollary \ref{Th:TorsionfreeBigFish} can be generalized to all purely
infinite simple unital Leavitt path algebras.  Rephrased, we seek to
show either

\medskip

(1) that there exist non-isomorphic purely
infinite simple Leavitt path algebras  $L(E),L(F)$ over finite
graphs $E,F$ such that ${\mathcal F}_{[1]}(E)\equiv {\mathcal F}_{[1]}(F)$, for which the signs of
$\mbox{det}(I-A_E^t)$ and $\mbox{det}(I-A_F^t)$ are unequal,  or

\smallskip

(2)  that the sign of $\mbox{det}(I-A_E^t)$ plays no role in guaranteeing the
existence of an isomorphism between the purely infinite simple unital Leavitt path algebras
$L(E)$ and $L(F)$, for which ${\mathcal F}_{[1]}(E)\equiv {\mathcal F}_{[1]}(F)$.

\medskip

A key observation related to the analysis of the Classification
Question developed by the authors in the present paper and
\cite{AbAnhLouP} is that the graph operations we have already considered
cannot help us in this final step, because all of these graph
operations preserve flow equivalence on subshifts of finite type,
and thus preserve the sign of ${\rm det}(I-A_E^t)$. So it is clear
that to attain the final goal of classifying these kinds of Leavitt
path algebras using the unitary Franks pair $(K_0(L(E)),
[1_{L(E)}])$ as an invariant requires a completely new set of ideas
and strategies.

In the context of Cuntz-Krieger $C^*$-algebras, the irrelevance of the sign of the
determinant in the analogous Classification Question was shown by R{\o}rdam
\cite{Rordam}; and in the case of graph $C^*$-algebras, this irrelevance  is a
direct consequence of the Kirchberg-Phillips Classification Theorem
\cite{Kirch, Phil} and the computation of the K-theoretic invariant
for such a class of algebras (see e.g. \cite{R}). In this direction,
a useful tool is Cuntz's Theorem (presented by R{\o}rdam
\cite[Theorem 7.2]{Rordam}), whose adaptation to our context gives
the possibility of reducing the above situation to a single pair of
algebras. We describe the situation, following Cuntz's argument.

\begin{defi}
{\rm For any finite graph $E$ having vertices $v_1,\ldots,v_n$, such that
$L(E)$ is a purely infinite simple algebra and $v_n$ belongs to a cycle,
let $E_{-}$ be the graph whose incidence matrix and pictorial
representation is
$$
A_{E_{-}} \ = \ \left(
\begin{array}{cccccc}
  &   &  &  & 0 & 0 \\
  & A_E  &  &  & \vdots & \vdots \\
  &   &  &  & 0 & 0 \\
  &   &  &  & 1 & 0 \\
 0 & \cdots  & 0 & 1 & 1 & 1 \\
 0 & \cdots  & 0 & 0 & 1 & 1 \\
\end{array}
\right)
\qquad
E_{-} \ = \ \xy
(0,0)*[F**:<8pt>]++{\xy
    (0,10)*{};
    (0,-10)*{};
    (3,5)*{E};
    (20,0)*{}="vv";
    (10,7)*{}="vv1";
    (10,0)*{}="vv2";
    (10,-7)*{}="vv3";
    {\ar@{.} "vv1" ;"vv" };
    {\ar@{.} "vv2" ;"vv" };
    {\ar@{.} "vv3" ;"vv" };
\endxy };
(10,0)*{\bullet}="v";
(13,0)*{\scriptstyle v_n};
(20,0)*{\bullet}="v1";
(30,0)*{\bullet}="v2";
{\ar@/^/     "v" ;"v1" };
{\ar@/^/     "v1";"v"  };
{\ar@/^/     "v1";"v2" };
{\ar@/^/     "v2";"v1" };
{\ar@(ul,ur) "v1";"v1" };
{\ar@(ur,dr) "v2";"v2" };
\endxy.
$$

Specifically, if $E^0=\{v_1, \dots ,v_n\}$ where $v_n$ belongs to a cycle,
then the new graph $E_{-}$ has  $E_{-}^0=\{v_1, \dots ,v_n, v_{n+1}, v_{n+2}\}$,
while $E_{-}^1$ is the union of $E^1$ with six new edges: one  from
$v_n$ to $v_{n+1}$; one from $v_{n+1}$ to each of $v_n$,  $v_{n+1}$,
and $v_{n+2}$; and one from $v_{n+2}$ to each of $v_{n+1}$ and
$v_{n+2}$.
}
\end{defi}

 It is straightforward to show (using the Purely Infinite
Simplicity Theorem) that $L(E)$ is purely infinite simple unital if and only
if $L(E_{-})$ is purely infinite simple unital.  Furthermore, we have

\begin{prop}\label{differentsigns}
Let $E$ be a finite graph for which $L(E)$ is purely infinite
simple, and let $E_{-}$ be the graph defined above. Then
$$K_0(L(E_{-}))  \cong K_0(L(E)),  \ \mbox{ and }\
\mbox{det}(I-A^t_{E_{-}})= -\mbox{det}(I-A^t_{E}).$$
\end{prop}
\begin{proof}
For proving $K_0(L(E_{-}))  \cong K_0(L(E))$, notice that, by
\cite[Theorem 3.5]{AMFP}, the monoid  $V(L(E_{-}))$ is generated
by $v_1, \dots , v_{n+2},$ with relations
$$v_i=\sum\limits
_{j=1}^{n}A_e(i,j)v_j \  \ \ (1\leq i\leq n-1),$$
 together with
the three relations
$$v_n= v_{n+1} + \sum\limits
_{j=1}^{n}A_e(i,j)v_j, \
v_{n+1}=v_n+v_{n+1}+v_{n+2}, \ \mbox{ and } v_{n+2}=v_{n+1}+v_{n+2}.$$
 Since
$L(E_{-})$ is purely infinite simple we get $K_0(L(E_{-}))\cong
V(L(E_{-}))^*$  by \cite[Corollary 2.2]{AGP}. Thus $K_0(L(E_{-}))$
is the group generated by $[v_1], \dots , [v_{n+2}]$, with
relations
$$[v_i]=\sum\limits _{j=1}^{n}A_e(i,j)[v_j] \ \ \  (1\leq
i\leq n-1),$$
together with the three relations
 $$[v_n]=[v_{n+1}] + \sum\limits
_{j=1}^{n}A_e(i,j)[v_j], \  [v_n]=-[v_{n+2}], \  \mbox{and}  \  [v_{n+1}]=[0].$$
In particular, since $[v_{n+1}]=[0]$  we in fact have that
$[v_i]=\sum\limits_{j=1}^{n}A_e(i,j)[v_j]$ for all $1\leq i \leq n$ (i.e., {\it including} $n$ as well). This yields that the relations  between generators of  $K_0(L(E))$ remain the same when viewed as elements of $K_0(L(E_{-}))$, so that the inclusion map    $K_0(L(E)) \mapsto K_0(L(E_{-}))$  is a group homomorphism.  The equations  $[v_{n+1}]=[0]$ and $[v_{n+2}]=-[v_n]$ show that this inclusion is actually a surjection,  so
that  $K_0(L(E_{-}))$ is isomorphic to $K_0(L(E))$.

(We note that the isomorphism from $K_0(L(E))$ to $K_0(L(E_{-}))$ given here does not necessarily take $[1_{L(E)}]$ to $[1_{L(E_{-})}]$, since in general we need not have  $[v_{n+1}] + [v_{n+2}]=[0]$ in $K_0(L(E_{-}))$.)

With respect to the determinants, the result is an elementary
computation.
\end{proof}

As a specific, important example, consider the Leavitt path algebras $L({\textbf{2}})$
and $L({\textbf{2}_{-}})$, where $\textbf{2}$ and $\textbf{2}_{-}$
are the graphs with incidence matrices
$$
A_{\textbf{2}}=\left(
\begin{array}{cc}
  1& 1   \\
  1& 1  \\
\end{array}
\right) \  \mbox{ and } \
A_{\textbf{2}_{-}}=\left(
\begin{array}{cccc}
 1 & 1 & 0 & 0  \\
 1 & 1 & 1 & 0  \\
 0 & 1 & 1 & 1  \\
  0& 0 & 1 & 1  \\
\end{array}
\right).
 $$
Pictorially, these graphs are given by
$$
\textbf{2} \  = \ \ \ \ \xymatrix{
  \bullet^{v_1} \ar@(dl,ul)[] \ar@/^/[r]
& \bullet^{v_2} \ar@(ur,dr)[] \ar@/^/[l] }
$$
\vspace{.2truecm}

and

$$
\textbf{2}_{-} \  =  \ \ \ \
\xymatrix{
  \bullet^{v_1} \ar@(dl,ul)[] \ar@/^/[r]
& \bullet^{v_2} \ar@(ul,ur)[] \ar@/^/[r] \ar@/^/[l]
& \bullet^{v_3} \ar@(ul,ur)[] \ar@/^/[r] \ar@/^/[l]
& \bullet^{v_4} \ar@(ur,dr)[] \ar@/^/[l]
}
$$

Notice that
$$(K_0(L({\textbf{2}})),[1_{L({\textbf{2}})}],
\mbox{det}(I-A_{\textbf{2}}^t))=(\{0\},0,-1), \ \ \mbox{while}$$
$$(K_0(L({\textbf{2}_{-}})),[1_{L({\textbf{2}_{-}})}],
\mbox{det}(I-A_{\textbf{2}_{-}}^t))=(\{0\},0,1).$$

Now consider the standard representations of $L({\textbf{2}})$ and
$L({\textbf{2}_{-}})$ in $R=\mbox{End}_K(V)$, where $V$ is a
$K$-vector space of countable dimension with basis $\{ v_i\}_{i\geq
1}$. (For a description of this process, see \cite[page 8]{R}.)  Let
$\textbf{u}\in R$ be the endomorphism defined by the rule
$\textbf{u}(v_i)=\delta_{1,i}v_1$ . Let $\mathcal{E}_2$ be the
subalgebra of $R$ generated by $L({\textbf{2}})$ and $\textbf{u}$,
and similarly let $\mathcal{E}_{2_{-}}$ be the subalgebra of $R$
generated by $L({\textbf{2}_{-}})$ and
$\textbf{u}$.\vspace{.2truecm}

\noindent \textbf{Hypothesis:} \emph{There exists a $K$-algebra
isomorphism $\tau: L(${\rm \textbf{2}}$)\rightarrow
L(${\rm \textbf{2}}$_{-})$ which extends to an isomorphism $T:
\mathcal{E}_2\rightarrow \mathcal{E}_{2_{-}}$ such that
$T(\textbf{u})=\textbf{u}$.}\vspace{.2truecm}

Using the argument presented in \cite[Theorem 7.2]{Rordam}, it is
long but straightforward to show that

\begin{theor}\label{Th:Cuntz reduction}
If the Hypothesis holds, then for any finite graph $E$ such that
$L(E)$ is a purely infinite simple Leavitt path algebra,  there is
a Morita equivalence $L(E)\sim_M L(E_{-})$.
\end{theor}

Therefore, as a consequence of Corollary \ref{determinantabsval} and
Proposition \ref{differentsigns} we would then have

\begin{theor}\label{Th:MoritaFish}
If the Hypothesis holds, then $K_0$ precisely classifies purely infinite
simple unital Leavitt path algebras up to Morita equivalence.
\end{theor}
\begin{proof}
Let $E,G$ be finite graphs for which $L(E)$ and $L(G)$ are purely infinite
simple and $K_0(L(E)) \cong K_0(L(G))$.  By Corollary \ref{determinantabsval},
either $\det(I-A_E^t) = \det(I-A_G^t)$, or $\det(I-A_E^t) = -\det(I-A_G^t)$.
In the first case, we have $\mathcal{F}_{det}(L(E)) \equiv \mathcal{F}_{det}(L(G))$,
so Theorem \ref{Franksinvariantsufficient} gives Morita equivalence.  Otherwise,
we have $\mathcal{F}_{det}(L(E_{-})) \equiv \mathcal{F}_{det}(L(G))$, so by
Theorems \ref{Th:Cuntz reduction} and \ref{Franksinvariantsufficient}, we get
$$L(E) \sim_M L(E_{-}) \sim_M L(G),$$ and the theorem follows.
\end{proof}

Following the same strategy as before, we now push this Morita equivalence
result to yield isomorphisms by applying Theorem \ref{Th:morita_iso}.  In order to
do so, we will need another graph construction.

\begin{defi}
{\rm For any finite graph $E$ having vertices $v_1,\ldots,v_n$, such that
$L(E)$ is a purely infinite simple algebra and $v_n$ belongs to a cycle,
let $E_{1-}$ be the graph whose incidence matrix and pictorial
representation is
$$
A_{E_{1-}} \ = \ \left(
\begin{array}{ccccccc}
   &         &   &   & 0      & 0      & 0      \\
   & A_E     &   &   & \vdots & \vdots & \vdots \\
   &         &   &   & 0      & 0      & 0      \\
   &         &   &   & 1      & 0      & 0      \\
 0 & \cdots  & 0 & 1 & 1      & 1      & 0      \\
 0 & \cdots  & 0 & 0 & 1      & 1      & 0      \\
 0 & \cdots  & 0 & 1 & 0      & 0      & 0      \\
\end{array}
\right)
\qquad
E_{1-} \ = \ \xy
(0,0)*[F**:<8pt>]++{\xy
    (0,10)*{};
    (0,-10)*{};
    (3,5)*{E};
    (20,0)*{}="vv";
    (10,7)*{}="vv1";
    (10,0)*{}="vv2";
    (10,-7)*{}="vv3";
    {\ar@{.} "vv1" ;"vv" };
    {\ar@{.} "vv2" ;"vv" };
    {\ar@{.} "vv3" ;"vv" };
\endxy };
(10,0)*{\bullet}="v";
(13,0)*{\scriptstyle v_n};
(20,0)*{\bullet}="v1";
(30,0)*{\bullet}="v2";
(13,-8)*{\bullet}="v3";
{\ar@/^/     "v" ;"v1" };
{\ar@/^/     "v1";"v"  };
{\ar@/^/     "v1";"v2" };
{\ar@/^/     "v2";"v1" };
{\ar@(ul,ur) "v1";"v1" };
{\ar@(ur,dr) "v2";"v2" };
{\ar         "v3";"v"  };
\endxy.
$$
}
\end{defi}

Immediately, we note that $E_{-} = \left(E_{1-}\right)_{\backslash v_{n+3}}$.
It is again straightforward to show (using the Purely Infinite Simplicity
Theorem) that $L(E)$ is purely infinite simple if and only
if $L(E_{1-})$ is purely infinite simple.  Furthermore, we have

\begin{prop}\label{1differentsigns}
Let $E$ be a finite graph for which $L(E)$ is purely infinite
simple, and let $E_{1-}$ be the graph defined above. Then
$$\mathcal{F}_{[1]}(L(E_{1-})) \equiv \mathcal{F}_{[1]}(L(E))
\ \mbox{ and}\  \mbox{det}(I-A^t_{E_{1-}})
= -\mbox{det}(I-A^t_{E}).$$
\end{prop}
\begin{proof}
Since $E_{-} = \left(E_{1-}\right)_{\backslash v_{n+3}}$, we get
from Lemma \ref{samedet} that $\det(I-A^t_{E_{1-}})
=\det(I-A^t_{E_{-}}) = -\det(I-A^t_E)$. We follow a  strategy similar to the one used in the proof of Proposition \ref{differentsigns} to conclude that the inclusion map is an isomorphism
between $K_0(L(E))$ and $K_0(L(E_{1-}))$.  Specifically, by again using
\cite[Theorem 3.5]{AMFP}, the monoid  $V(L(E_{1-}))$ is
generated by $v_1, \dots , v_{n+3}$ with relations $v_i=\sum\limits
_{j=1}^{n}A_e(i,j)v_j$ ($1\leq i\leq n-1$), together with the four
relations
$$v_n= v_{n+1} + \sum\limits _{j=1}^{n}A_e(i,j)v_j, \
v_{n+1}=v_n+v_{n+1}+v_{n+2}, \ v_{n+2}=v_{n+1}+v_{n+2},  \ v_{n+3}=v_n.$$
  Again, we apply \cite[Corollary 2.2]{AGP} to get the
isomorphism $K_0(L(E_{1-})) \cong V(L(E_{1-}))^*$, and note that
$[v_{n+2}]= -[v_n]$, $[v_{n+3}]=[v_n]$ and $[v_{n+1}]=[0]$, so that
$[v_i]=\sum\limits _{j=1}^{n}A_e(i,j)[v_j]$ ($1\leq i\leq n$).
Hence, the map
\begin{eqnarray*}
 {[v_i]}     & \mapsto & [v_i]\  \ \  (1 \leq i\leq n) \\
 {[v_{n+2}]} & \mapsto & -[v_n] \\
 {[v_{n+3}]} & \mapsto & [v_n]
\end{eqnarray*}
defines an isomorphism $\varphi$ from $K_0(L(E_{1-}))$ to
$K_0(L(E))$.  In addition,
$$\varphi([1_{L(E_{1-})}]) =
[1_{L(E)}] + [v_{n+1}] + [v_{n+2}] + [v_{n+3}] = [1_{L(E)}] + [v_n]
- [v_n] = [1_{L(E)}],$$
which yields the equivalence $\mathcal{F}_{[1]}(L(E_{1-})) \equiv \mathcal{F}_{[1]}(L(E))$.
\end{proof}

In particular, for any finite graph $E$ with $L(E)$ purely infinite simple, we
now have a construction of another graph $E_{1-}$ which shares its unitary Franks
pair, but differs in $\mbox{sgn}(\det(I-A_E^t))$.  We now assume the Hypothesis, and
analyze the consequences for isomorphisms.

\begin{prop}\label{1morita_equiv}
If the Hypothesis holds, then for any finite graph $E$ such that
$L(E)$ is a purely infinite simple Leavitt path algebra,  there is
a Morita equivalence $L(E)\sim_M L(E_{1-})$.
\end{prop}
\begin{proof}
Applying Theorem \ref{Th:Cuntz reduction} and Proposition \ref{sourceelimnationprop},
we get $$L(E) \sim_M L(E_{-}) = L((E_{1-})_{\backslash v_{n+3}}) \sim_M L(E_{1-}).$$
\end{proof}

Combining the equivalence of the unitary Franks pair from Proposition
\ref{1differentsigns} with the Morita equivalence from Proposition \ref{1morita_equiv},
we are in position to apply Theorem \ref{Th:morita_iso} and obtain the following key connecting result, one which allows us to cross the
``determinant gap".

\begin{prop}\label{1iso}
If the Hypothesis holds, then for any finite graph $E$ such that
$L(E)$ is a purely infinite simple Leavitt path algebra,  there is
an isomorphism $L(E)\cong L(E_{1-})$.
\end{prop}

As a consequence, we obtain

\begin{theor}\label{HypotheticalBigFish}
If the Hypothesis holds, then $\mathcal{F}_{[1]}$ precisely classifies purely infinite
simple unital Leavitt path algebras up to isomorphism.
\end{theor}
\begin{proof}
Let $E,G$ be finite graphs for which $L(E)$ and $L(G)$ are purely infinite
simple and $\mathcal{F}_{[1]}(L(E)) \equiv \mathcal{F}_{[1]}(L(G))$.  By Corollary
\ref{determinantabsval}, either $\det(I-A_E^t) = \det(I-A_G^t)$, or
$\det(I-A_E^t) = -\det(I-A_G^t)$.  In the first case, we have
$\mathcal{F}_{3}(L(E)) \equiv \mathcal{F}_{3}(L(G))$,
so Corollary \ref{Th:Almost BigFish} gives the desired isomorphism.  Otherwise,
we have $\mathcal{F}_{3}(L(E_{1-})) \equiv \mathcal{F}_{3}(L(G))$, so by
Proposition \ref{1iso} and Corollary \ref{Th:Almost BigFish}, we get
$$L(E) \cong L(E_{1-}) \cong L(G),$$ and the theorem follows.
\end{proof}

\section{Some general isomorphism and Morita equivalence results for Leavitt
path algebras}\label{Extra}

In Section \ref{Morita Equiv} we presented four results regarding
Morita equivalences between Leavitt path algebras.  These four
specific results were precisely those which we needed to achieve the
first main result of this article, Theorem
\ref{Franksinvariantsufficient}.  In the final section of this
article, we present a number of similarly-flavored results which we
believe are of interest in their own right.    Along the way we will
give generalizations of Propositions  \ref{sourceelimnationprop},
\ref{expansionprop}, and \ref{insplittingprop} to wider classes of
graphs.

Information about various topics presented in this section pertaining to Leavitt path algebras (e.g. the $\Z$-grading on $L(E)$, and the natural action as automorphisms of $K^*$ on $L_K(E)$) can be found in \cite{AbAnhLouP}.  Information about Morita equivalence for not-necessarily-unital rings can be found in \cite{GS}.

Here is the indicated generalization of Proposition \ref{sourceelimnationprop}.

\begin{prop}\label{P:erase sources}
Let $E$ be a row-finite graph with no sinks, let $v\in E^0$ be a
source, and let $E_{\backslash v}$ be the source elimination graph.
Then $L(E)$ and $L(E_{\backslash v})$ are Morita equivalent.
\end{prop}
\begin{proof}
By definition of $F=E_{\backslash v}$, it is clear that $F$ is a
(complete) subgraph of $E$. Thus, the $K$-algebra map defined by the
rule
$$\begin{array}{cccc}
\phi & L(F) & \longrightarrow & L(E) \\
 & w & \mapsto & w\\
 & e & \mapsto & e\\
     & e^* & \mapsto & e^*
\end{array}
$$
for every $w\in F^0$ and every $e\in F^1$, is a $\Z$-graded ring homomorphism
such that $\phi(w)\ne 0$ for every $w\in F^0$. Hence $\phi$ is
injective  by \cite[Lemma 1.1]{AbAnhLouP}.

Set $F^0=\{ w_i\}_{i\geq 1}$. For each $n\geq 1$ define
$e =\sum\limits_{i=1}^nw_i$. Then $\{ e_n\}_{n\geq 1}$ is a set of local
units for $L(F)$, and since $v$ is a source, $\phi
(L(F))=\bigcup\limits_{n\geq 1}e_nL(E)e_n$. Moreover, as
$r(s^{-1}(v))\subset F^0$, $E^0$ turns out to be the hereditary
saturated closure of $F^0$. Hence, we get
$L(E)=\bigcup\limits_{n\geq 1}L(E)e_nL(E)$. Thus, it is not
difficult to see that
$$(\sum\limits_{n\geq 1}e_nL(E)e_n, \ \sum\limits_{n\geq 1}L(E)e_nL(E), \
\sum_{n\geq 1}L(E)e_n, \  \sum_{n\geq 1}e_nL(E))$$ is a (surjective)
Morita context for the rings $L(E)$ and $L(F)$, as desired.
\end{proof}

The following definition are borrowed from \cite[Section 4]{Flow}.

\begin{defis}\label{outdelaydefinition}
{\rm Let $E = ( E^0, E^1, r, s )$ be a row-finite graph. A map $d_s : E^0
\cup E^1 \rightarrow \N \cup \{ \infty \}$ such that
\begin{enumerate}
\item if $w \in E^0$ is not a sink then $d_s (w) = \sup \{ d_s (e)
\mid s(e)=w \}$, and \item if $d_s ( x ) = \infty$ for some $x$, then
$x$ is a sink
\end{enumerate}

\noindent is called a {\em Drinen source-vector}. Note that only
vertices are allowed to have an infinite $d_s$-value. From this
data we construct a new graph $d_s(E)$ as follows: Let
\begin{align*}
d_s (E)^0 &= \{ v^i \mid v \in E^0 , 0 \le i \le d_s (v) \}, \text{ and } \\
d_s (E)^1 &= E^1 \cup \{ f(v)^i \mid 1 \le i \le d_s (v)\} ,
\end{align*}

\noindent and for $e \in E^1$ define $r_{d_s (E)} (e) = r(e)^0$ and
$s_{d_s (E)} (e) = s(e)^{d_s (e)}$. For $f (v)^i$ define $s_{d_s (E)} (
f(v)^{i} ) = v^{i-1}$ and $r_{d_s (E)} ( f(v)^i ) = v^i$.  The
resulting directed graph $d_s (E)$ is called the {\em out-delayed
graph of $E$} for the Drinen source-vector $d_s$.
}
\end{defis}

In the out-delayed graph the original vertices correspond to those
vertices with superscript $0$.  Intuitively, the edge $e \in E^1$ is ``delayed"
from leaving $s(e)^0$ and arriving at $r(e)^0$ by a path of length
$d_s (e)$.

\begin{theor}\label{outdelaysme}
Let $K$ be an infinite field.  Let $E$ be a row-finite graph and let $d_s : E^0 \cup E^1
\rightarrow \N \cup \{ \infty \}$ be a Drinen source-vector. Then
$L ( d_s ( E ) )$ is Morita equivalent to $L(E)$.
\end{theor}
\begin{proof}
The argument is essentially the same as in the proof of
\cite[Theorem 4.2]{Flow}, except for the proof of the injectivity of the map
$\pi$, and the proof of the Morita equivalence of $L(E)$ and
$L(d_s(E))$. We include the whole argument for the sake of
completeness.

Given $e \in E^1$ and $v \in E^0$, define $Q_v = v^0$, and define $T_e$ by setting
\[
T_e = f(s(e))^1 \ldots f(s(e))^{d_s (e)} e \text{ if } d_s (e)
\neq 0,   \ \ \text{ and } \ \  T_e = e \text{ otherwise.}
\]
\noindent We claim that $\{ T_e , Q_v \mid e \in E^1 , v \in E^0 \}$
is an $E$-family in $L( d_s ( E ) )$. The $Q_v$'s are nonzero mutually
orthogonal idempotents since the $v^0$'s are. The elements $T_e$ for
$e \in E^1$ clearly satisfy $T_e^*T_f=0$ whenever $e\ne f$, because
they consist of sums of elements with the same property. For $e \in
E^1$ it is routine to check that $T_e^* T_e = Q_{r(e)}$.

If $v \in E^0$ is not a sink, then $d_s (v) < \infty$. If $d_s (v)
=0$, then we certainly have $Q_v = \sum\limits_{\{s(e)=v\}} T_e
T_e^*$. Otherwise, for $0 \le j \le d_s (v)-1$ we have
\begin{equation} \label{gantletrels}
v^j = \sum\limits_{\{s(e)=v , d_s (e)=j\}} e e^* + f(v)^{j+1}
v^{j+1} (f(v)^{j+1})^* ,
\end{equation}

\noindent and since we must have some edges with $s(e)=v$ and $d_s
(e) = d_s (v)$ we have
\begin{equation} \label{endcond}
v^{d_s (v)} = \sum\limits_{\{s(e)=v , d_s (e) = d_s (v)\}} e e^* .
\end{equation}

\noindent Using (\ref{gantletrels}) recursively and
(\ref{endcond}) when $j=d_s (v)-1$ we see that
\[
Q_v = v^0 = \sum_{\{s(e)=v,d_s(e)=0\}} T_e T_e^* + \cdots +
\sum_{\{s(e)=v,d_s(e)=d_s(v)\}} T_e T_e^* = \sum_{\{s(e)=v\}} T_e
T_e^* ,
\]
and this establishes our claim.

So by the Universal Homomorphism Property of $L_K(E)$ there is a $K$-algebra homomorphism
$$\pi : L_K(E) \rightarrow L_K( d_s ( E ) )$$
which takes $e$ to $T_e$
and $v$ to $Q_v$.

 Let $\alpha^{d_s(E)}$ denote the
$K$-action as automorphisms on $L_K( d_s ( E )$ satisfying, for each $t \in K^* = K \setminus \{0\}$,
$$
\alpha^{d_s(E)}_t (e) = t e , \  \  \alpha^{d_s(E)}_t({f(v)^i}) = {f(v)^i} \text{ for } 1
\le i \le d_s(v), \ \text{ and } \ \alpha^{d_s(E)}_t ({v^i}) = {v^i} \text{ for
} 0 \le i \le d_s(v).
$$

We now establish the injectivity of $\pi$ for all fields $K$. It is
straightforward to check that $\pi \circ \tau^E_t =
\alpha^{d_s(E)}_t \circ \pi$ for each $t \in K^*$, where $\tau^E$ is
the  standard action of $K$ on $L_K(E)$ given in \cite[Definition
1.5]{AbAnhLouP}.   Since $K$ is infinite,  it follows from
\cite[Theorem 1.8]{AbAnhLouP} that $\pi$ is injective.

Now enumerate the vertices $v^0\in d_s(E)^0$ by $\{ v_i^0\mid i\geq
1\}$, and define $\{ e_n\}_{n\geq 1}$ by
$e_n=\sum\limits_{i=1}^{n}v_i^0$. Then $\{ e_n\}_{n\geq 1}$ is an
ascending chain of idempotents in $L(d_s(E))$. Let $A=\bigcup\limits
_{n\geq1} e_n L(d_s(E))e_n$ be the subalgebra of $L(d_s(E))$ with set of  local units $\{ e_n\}_{n\geq 1}$. The same argument as in
\cite[Theorem 4.2]{Flow} shows that $A=\pi (L(E))$, which is
isomorphic to $L(E)$ by the previously established injectivity of
$\pi$. Also, the same argument as in \cite[Lemma 2.4]{APS} shows
that $A$ is Morita equivalent to the ideal
$$I=\bigcup\limits _{n\geq 1}L(d_s(E))e_nL(d_s(E))=
\sum\limits _{v\in E^0}L(d_s(E))v^0 L(d_s(E))=I(\{ v^0\mid v\in
E^0\}).$$ Since $d_s(E)^0$ is the hereditary saturated closure of
$\{ v^0\mid v\in E^0\}$, $I=L(d_s(E))$ by \cite[Lemma 2.1]{APS},
whence the result holds.
\end{proof}

Let $E$ be a row-finite graph, and let $v\in E^0$.  We define the
following  Drinen source-vector:
\begin{enumerate}
\item For every $w\in E^0\setminus \{ v\}$, $d_s(w)=0$, while
$d_s(v)=1$.\item For every $f\in E^1\setminus s^{-1}(v)$,
$d_s(f)=0$, while $d_s(e)=1$ for any $e\in s^{-1}(v)$.
\end{enumerate}
Then it is straightforward to see that
$$ E_{v} = d_s(E).$$
In other words, the out-delay graph $d_s(E)$ related to this particular  Drinen source-vector is precisely the expansion graph
$E_{v}$.   With this observation, Theorem \ref{outdelaysme} then
immediately yields this more general version of Proposition
\ref{expansionprop}

\begin{corol}\label{C:expand==MorEquiv}
Let $K$ be an infinite field.  Let $E$ be a row-finite graph, and $v\in E^0$.   Then $L(E_{v})$ is Morita equivalent to $L(E)$.
\end{corol}

We again borrow  definitions from \cite[Section 4]{Flow}.

\begin{defis}\label{indelaydefinition}
{\rm Let $E = ( E^0, E^1, r, s )$ be a row-finite graph. A map $d_r : E^0 \cup E^1
\rightarrow \N \cup \{ \infty \}$ satisfying
\begin{enumerate}
\item if $w$ is not a source then $d_r (w) = \sup \, \{ d_r (e)
\mid r(e)=w \}$, and
 \item if $d_r ( x ) = \infty$  then $x$ is either
a source or receives infinitely many edges
\end{enumerate}

\noindent is called a {\em Drinen range-vector}. We construct a
new graph $d_r(E)$, called the {\em in-delayed graph of $E$} for
the Drinen range-vector $d_r$, as follows:
\begin{align*}
d_r (E)^0 &= \{ v_i \mid v \in E^0 , 0 \le i \le d_r (v) \},  \\
d_r (E)^1 &= E^1 \cup \{ f(v)_i \mid 1 \le i \le d_r (v) \}.
\end{align*}

\noindent For $e \in d_r (E)^1$ with $e \in E^1$ we define $r_{d_r (E)} (e) = r(e)_{d_r
(e)}$ and $s_{d_r (E)} (e) = s(e)_0$. For $f\in d_r (E)^1$ of the form $f = f (v)_i$ we define $s_{d_r (E)} (
f(v)_i ) = v_{i}$ and $r_{d_r (E)} ( f(v)_i ) = v_{i-1}$.
}
\end{defis}

\begin{theor} \label{indelayme}
Let $K$ be an infinite field.  Let $E$ be a row-finite graph and let $d_r : E^0 \cup E^1 \rightarrow \N
\cup \{ \infty \}$ be a Drinen range-vector. Then $L( d_r ( E )
)$ is Morita equivalent to $L(E)$.
\end{theor}
\begin{proof}
The proof  is essentially identical to the proof of \cite[Theorem 4.5]{Flow},
using  arguments analogous to those used in  the proof of Theorem \ref{outdelaysme}.
\end{proof}

We now give an additional condition on the previously defined notion
of an {\it in-split graph} (see the notation presented in Definitions
\ref{def_insplit}).
\begin{defi}
{\rm  Let $E$ be a graph, let $\mathcal{P}$ be a partition of $E^1$, and let
$m$ be as described in Definitions \ref{def_insplit}. $\mathcal{P}$ is
called {\em proper} if for every vertex $v$ which is a sink we have
$m(v)=0$ or $m(v)=1$. (That is, $\mathcal{P}$ if proper if $\mathcal{P}$ does not in-split at a sink.)}
\end{defi}

To relate the Leavitt path algebra of a graph to the Leavitt path
algebras of its in-splittings we use a variation of the method
introduced in \cite[Section 4.2]{d}: If $E_r ( \mathcal{P} )$ is the
in-split graph formed from $E$ using the partition $\mathcal{P}$
then we may define a Drinen range-vector $d_{r , \mathcal{P}} : E^0
\cup E^1 \rightarrow \N \cup \{ \infty \}$ by $d_{r , \mathcal{P}} (
v ) = m ( v )-1$ if $m(v) \ge 1$ and $d_{r, \mathcal{P}}(v)=0$
otherwise. For $e \in \mathcal{E}^{r(e)}_i$ we put $d_{r ,
\mathcal{P}} ( e ) = i-1$. Hence, if $v$ receives $n \ge 2$ edges
then we create an in-delayed graph in which $v$ is given delay of
size $m(v)-1$ and all edges with range $v$ are given a delay one
less than their label in the partition of $r^{-1} (v)$. If $v$ is a
source or receives only one edge then there is no delay attached to
$v$.

\begin{theor} \label{insplitindelay}
Let $K$ be an infinite field.  Let $E$ be a row-finite graph, $\mathcal{P}$ a partition of $E^1$,
$E_r ( \mathcal{P} )$ the in-split graph formed from $E$ using
$\mathcal{P}$ and $d_{r , \mathcal{P}} : E^0 \cup E^1 \rightarrow
\N \cup \{ \infty \}$  the Drinen range-vector defined as above.
Then $L( E_r ( \mathcal{P} ) ) \cong L ( d_{r ,\mathcal{P}} (E) )$
if and only if $\mathcal{P}$ is proper.
\end{theor}
\begin{proof}
The proof is analogous to the proof of \cite[Theorem 5.3]{Flow}, using the
arguments of the proof of Theorem \ref{outdelaysme}.
\end{proof}

Applying Theorem \ref{insplitindelay} and Theorem \ref{indelayme},
we get the following analog to \cite[Corollary 5.4]{Flow}, which in
turn gives a generalization of Proposition \ref{insplittingprop}.
(Note that the hypotheses of Proposition \ref{insplittingprop}
include that $E$ contains no sinks, so that every partition of $E^1$
is vacuously proper.)

\begin{corol} \label{insplitme}
Let $K$ be an infinite field.  Let $E$ be a row-finite graph, $\mathcal{P}$ a partition of $E^1$
and $E_r ( \mathcal{P} )$ the in-split graph formed from $E$ using
$\mathcal{P}$.  Then $L ( E_r ( \mathcal{P} ) )$ is Morita
equivalent to $L ( E )$ if and only if $\mathcal{P}$ is proper.
\end{corol}

 Having now finished the work of extending Propositions  \ref{sourceelimnationprop},
\ref{expansionprop}, and \ref{insplittingprop} to wider classes of
graphs, we note here for completeness that, by \cite[Theorem 2.8]{AbAnhLouP},  Proposition \ref{outsplittingprop} extends verbatim to row-finite graphs as well.

\medskip

We conclude this article by analyzing the relationship between the
Leavitt path algebra $L(E)$ of a graph $E$ and the Leavitt path
algebra  $L(E^t)$ of its transpose graph $E^t$.   An easy example shows that in general these
two algebras need not be Morita equivalent.  For instance, if $E$ is
the graph
$$ E \ \ = \ \ \xymatrix{ {\bullet} & {\bullet} \ar[r] \ar[l]  & {\bullet}}$$
then
$$E^t \ \ =  \ \ \xymatrix{ {\bullet} \ar[r]& {\bullet}    & {\bullet} \ar[l] }$$
By \cite[Proposition 3.5]{AAS1} we get that $L_K(E)\cong {\rm M}_2(K) \oplus {\rm
M}_2(K)$, while $L_K(E^t)\cong {\rm M}_3(K)$; these two algebras are
not Morita equivalent.

Indeed, we can find a finite graph $E$ having neither sinks nor sources for which $L(E)$ and $L(E^t)$ are not Morita equivalent.  Specifically, consider the graph $E$
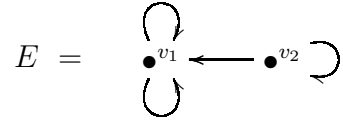
$$
{
\def\labelstyle{\displaystyle}{E} \ = \ \quad
\xymatrix{ \bullet^{v_1}\dloopr{}\uloopr{} &  \bullet^{v_2}
\ar[l]
\uloopd{} \ar[l] }}
$$

whose transpose graph $E^t$ is

$$
{
\def\labelstyle{\displaystyle}{E^t} \ = \ \quad
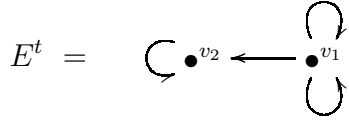
\xymatrix{ \bullet^{v_2}\dloopd{} &  \bullet^{v_1} \ar[l]
\uloopr{}\dloopr{} \ar[l] }}
$$

Then $v_1\in E^0$ generates the unique proper graded
two-sided
ideal of $L_K(E)$, and the quotient ring $L_K(E)/\langle v_1\rangle$ is isomorphic to $K[x,x^{-1}]$. Thus, $L(E)$ has no purely infinite simple
unital quotients. Since $E$ contains loops, \cite[Theorem
2.8]{APsr}
implies that the stable rank $\mbox{sr}(L(E))$ equals $2$. On the other hand, $v_2\in
(E^t)^0$ generates a proper graded two-sided ideal in $L(E^t)$, whose
quotient ring $L(E^t)/\langle v_2\rangle$ is isomorphic to the Leavitt algebra $L_K(1,2)$. Thus
$\mbox{sr}(L(E^t))=\infty$ by \cite[Theorem 2.8]{APsr}. But the
stable rank is a Morita invariant for unital Leavitt path
algebras
of row-finite graphs \cite[Remark 3.4(1)]{APsr}, so that $L(E)$
and
$L(E^t)$ cannot be Morita equivalent.

However, in contrast to the previous two examples, we get the following consequence of Theorem \ref{Th:Almost BigFish}.

\begin{prop}\label{C:transpose graph}
If $E$ is a finite graph without sources such that $L(E)$ is a
purely infinite simple Leavitt path algebra, then $L(E)$ and
$L(E^t)$ are Morita equivalent.
\end{prop}
\begin{proof}
There is an isomorphism $coker(I-A_E^t) \cong coker(I-A_{E^t}^t) =
coker(I-A_E)$, since the Smith normal forms of $I-A_E^t$  and
$I-A_E$ are equal. Furthermore, cofactor expansions clearly yield
${\rm det}(I - A_E^t) = {\rm det}(I - A_{E^t}^t) = {\rm det}(I -
A_{E}).$ Thus we have ${\mathcal F}_{det}(E) \equiv {\mathcal
F}_{det}(E^t)$.   By Lemma \ref{graphconditions} we have that $E$ is
irreducible, essential, and nontrivial.  But these three conditions
on a graph are easily seen to pass to the transpose graph $E^t$, so
that (again by Lemma \ref{graphconditions}) we have that $L(E^t)$ is
purely infinite simple.   Thus Theorem
\ref{Franksinvariantsufficient} applies to yield the result.
\end{proof}

The result of Proposition \ref{C:transpose graph} does not extend to
isomorphisms, as the following example demonstrates.

\begin{exem}\label{Ex:transp not iso}
{\rm Consider the graph $E$ whose incidence matrix is
$$
A_E=\left(
\begin{array}{ccc}
1  & 1 & 1  \\
 0 &  0 &  1 \\
 1 & 0 & 0 \\
\end{array}
\right).
$$
Then $E$ is a graph with no sources, for which $L(E)$ is purely
infinite simple.   It is not hard to show (see e.g.
\cite[pp. 67-68]{T1}) that $K_0(L(E)) = \Z_2$, and $[1_E] = [1]$ in
$\Z_2$.   On the other hand, a similarly easy computation yields
that $K_0(L(E^t)) = \Z_2$ as well, but $[1_{E^t}] = [0]$ in $\Z_2$.
Since, as noted above, an isomorphism between Leavitt path algebras
yields an equivalence of the corresponding unitary Franks pairs, we
conclude that $L(E)\not\cong L(E^t)$.}
\end{exem}

\section*{Acknowledgments}

The authors thank the referee for a very thorough and careful reading of the original version of this article.   Part of this work was done during a visit of the third author to the
Department of Mathematics of the University of Colorado at Colorado
Springs. The third author thanks both the Department of
Mathematics of UCCS and the Department of Mathematics of The
Colorado College for their kind hospitality.


\begin{thebibliography}{00}

\bibitem{AA1} \textsc{G. Abrams, G. Aranda Pino}, The Leavitt path algebra of a graph,
\emph{J. Algebra} \textbf{293} (2005), 319--334.

\bibitem{AA2} \textsc{G. Abrams, G. Aranda Pino}, Purely infinite simple Leavitt path
algebras, \emph{J. Pure Appl. Algebra}, \textbf{207} (2006), 553--563.

\bibitem{AAS1} \textsc{G. Abrams, G. Aranda Pino, M. Siles Molina},  Finite dimensional
Leavitt path algebras, \emph{J. Pure Appl. Algebra}  \textbf{209}(3) (2007), 753--762.

\bibitem{AbAnhLouP} \textsc{G. Abrams, P. N. \'{A}nh, A. Louly, E. Pardo}, The
classification question for Leavitt path algebras, \emph{J. Algebra} \textbf{320}
(2008), 1983--2026.

\bibitem{AS}  \textsc{G. Abrams, C. Smith},  Explicit isomorphisms between purely infinite simple Leavitt path algebras, in preparation.

\bibitem{AGP} \textsc{P.Ara, K. Goodearl, E. Pardo}, $K_0$ of purely infinite
simple regular rings, \emph{K-Theory}  \textbf{26} (2002), 69--100.

\bibitem{AMFP} \textsc{P. Ara, M.A. Moreno, E. Pardo}, Nonstable K-Theory for graph algebras,
\emph{Algebr. Represent. Theory} \textbf{10} (2007), 157--178.

\bibitem{APsr} \textsc{P. Ara, E. Pardo}, Stable rank of Leavitt
path algebras, \emph{Proc. Amer. Math. Soc.} \textbf{136}(7)  (2008),
 2375--2386.

\bibitem{APS} \textsc{G. Aranda Pino, E. Pardo, M. Siles Molina}, Exchange Leavitt
path algebras and stable rank, \emph{J. Algebra} \textbf{305}
(2006), 912--936.

\bibitem{Flow} \textsc{T. Bates, D. Pask}, Flow equivalence of graph
algebras, \emph{Ergodic Theory Dynam. Systems} \textbf{24} (2004), 367--382.

\bibitem{B-F} \textsc{R. Bowen, J. Franks},  Homology for
zero-dimensional nonwandering sets, \emph{Ann. of Math.} \textbf{106}
(1977), 73--92.

\bibitem{CK}  \textsc{J. Cuntz, W. Krieger}, A class of C$^*$-algebras and topological Markov chains, \emph{Invent. Math.} \textbf{56} (1980), 251--268.

\bibitem{d} \textsc{D. Drinen}, Flow equivalence and graph groupoid
isomorphism, \emph{Dartmouth College} (2001), preprint.

\bibitem{Franks} \textsc{J. Franks}, Flow equivalence of subshifts
of finite type, \emph{Ergodic Theory Dynam. Systems} \textbf{4} (1984),
53--66.

\bibitem{GS}  \textsc{J.L. Garc\'{\i}a, J.J. Sim\'{o}n},  Morita equivalence for idempotent rings, \emph{J. Pure Appl. Algebra} \textbf{76}(1)  (1991), 39–-56.

\bibitem{Kirch}
\textsc{E. Kirchberg}, The classification of purely infinite C*-algebras using Kasparov theory,
preprint.

\bibitem{Huang2} \textsc{D. Huang}, Flow equivalence of reducible shifts of finite type,
\emph{Ergodic Theory Dynam. Systems} \textbf{14} (1994), 695--720.

\bibitem{Huang} \textsc{D. Huang}, Automorphisms of Bowen-Franks
groups of shifts of finite type, \emph{Ergodic Theory Dynam. Systems}
\textbf{21} (2001), 1113--1137.

\bibitem{L-M}\textsc{D. Lind, B. Marcus}, ``An Introduction to Symbolic Dynamics and
Coding''. Cambridge Univ. Press, Cambridge, 1995, Reprinted 1999
(with corrections). ISBN 0-521-55900-6.

\bibitem{P-S} \textsc{W. Parry, D. Sullivan}, A topological
invariant for flows on one-dimensional spaces, \emph{Topology}
\textbf{14} (1975), 297--299.

\bibitem{Phil}
\textsc{N.C. Phillips}, A classification theorem for nuclear purely infinite simple C*-algebras,
\emph{Doc. Math.} \textbf{5} (2000), 49--114.

\bibitem{R} \textsc{I. Raeburn}, ``Graph algebras". CBMS Regional Conference
Series in Mathematics \textbf{103},
American Mathematical Society, Providence, 2005. ISBN 0-8218-3660-9

\bibitem{Rordam} \textsc{M. R{\o}rdam}, Classication of
Cuntz-Krieger algebras, \emph{K-Theory} \textbf{9} (1995), 31--58.

\bibitem{T1} \textsc{M. Tomforde}, Structure of graph C$^*$-algebras and
generalizations, in ``Graph
algebras: bridging the gap between analysis and algebra", G. Aranda
Pino, F. Perera, M. Siles Molina (eds.), Universidad de M\'{a}laga
Press, 2007. ISBN 978-84-9747-177-0.

\bibitem{W} \textsc{R.F. Williams}, Classification of subshifts of
finite type, \emph{Ann. of Math.}  \textbf{98}(2) (1973), 120--153.

\end{thebibliography}
\end{document}